\documentclass[11pt,reqno]{amsart}
\usepackage{amssymb,mathrsfs,graphicx}
\usepackage{ifthen}
\usepackage{colortbl}
\definecolor{black}{rgb}{0.0, 0.0, 0.0}
\definecolor{red}{rgb}{1.0, 0.5, 0.5}

\usepackage{ifthen} 
\provideboolean{shownotes} 
\setboolean{shownotes}{true} 
\newcommand{\margnote}[1]{
\ifthenelse{\boolean{shownotes}}%
{\marginpar{\raggedright\tiny\texttt{#1}}}%
{}%
}
\newcommand{\hole}[1]{
\ifthenelse{\boolean{shownotes}}%
{\begin{center} \fbox{ \rule {.25cm}{0cm} \rule[-.1cm]{0cm}{.4cm}
\parbox{.85\textwidth}{\begin{center} \texttt{#1}\end{center}} \rule
{.25cm}{0cm}}\end{center}} {} }

\title[The derivation of Swarming models]{The derivation of Swarming models: Mean-Field Limit and Wasserstein distances}

\author[Carrillo]{Jos\'e Antonio Carrillo}
\address[Jos\'e Antonio Carrillo]{\newline Department of Mathematics \newline
Imperial College London, London SW7 2AZ, United Kingdom}
\email{carrillo@imperial.ac.uk}

\author[Choi]{Young-Pil Choi}
\address[Young-Pil Choi]{\newline Department of Mathematics \newline
Imperial College London, London SW7 2AZ, United Kingdom}
\email{young-pil.choi@imperial.ac.uk}

\author[Hauray]{Maxime Hauray}
\address[Maxime Hauray]{\newline Centre de Math\'ematiques et Informatique, \newline
    Universit\'e d'Aix-Marseille, Technop\^ole Ch\^ateau-Gombert, Marseille, France}
\email{maxime.hauray@univ-amu.fr}

\numberwithin{equation}{section}

\newtheorem{theorem}{Theorem}[section]

\newtheorem{corollary}{Corollary}[section]
\newtheorem{proposition}{Proposition}[section]
\newtheorem{remark}{Remark}[section]
\newtheorem{definition}{Definition}[section]

\newcommand{\bbr}{\mathbb R}
\newcommand{\R}{\mathbb R}

\newcommand{\bbp} {\mathbb P}
\newcommand{\PP} {\mathbb P}
\newcommand{\EE} {\mathbb E}

\newcommand{\N}{\mathbb N}
\newcommand{\e}{\varepsilon}
\newcommand{\ep}{\varepsilon}

\newcommand{\Pc}{\mathcal{P}_c}
\def\charf {\mbox{{\text 1}\kern-.30em {\text l}}}
\def\lip{\mathrm{Lip}}
\def\grad{\nabla}

\DeclareMathOperator{\supp}{supp}
\DeclareMathOperator{\dv}{div}

\begin{document}
\allowdisplaybreaks



\thanks{\textbf{Acknowledgments.}
JAC was partially supported by the project MTM2011-27739-C04-02
DGI (Spain) and 2009-SGR-345 from AGAUR-Generalitat de Catalunya.
JAC acknowledges support from the Royal Society by a Wolfson
Research Merit Award. YPC was supported by Basic Science Research
Program through the National Research Foundation of Korea funded
by the Ministry of Education, Science and Technology (ref.
2012R1A6A3A03039496). JAC and YPC were supported by Engineering
and Physical Sciences Research Council grants with references
EP/K008404/1 (individual grant) and EP/I019111/1 (platform grant).}

\begin{abstract}
    These notes are devoted to a summary on the mean-field limit
    of large ensembles of interacting particles with applications
    in swarming models. We first make a summary of the kinetic
    models derived as continuum versions of second order models
    for swarming. We focus on the question of passing from the
    discrete to the continuum model in the Dobrushin framework. We
    show how to use related techniques from fluid mechanics
    equations applied to first order models for swarming, also
    called the aggregation equation. We give qualitative bounds on
    the approximation of initial data by particles to obtain the
    mean-field limit for radial singular (at the origin) potentials
    up to the Newtonian singularity. We also show the propagation
    of chaos for more restricted set of singular potentials.
    \end{abstract}

\maketitle 

\tableofcontents

%
%
%
%
\section{Introduction}
In the last years, we have seen the development of a great deal of
different models in the biology, applied mathematics, and physics
literature to describe the collective behavior of individuals.
Here, individuals may mean animals (insects, fish, birds,...),
bacteria, and even robots. Most of these models involve the
nonlocal character of the interaction as a basic modelling pillar,
see for instance \cite{camazine,couzin,LLE,vicek}. In fact, one of
largest source of collective behavior models comes from control
engineering. There, the aim is to produce a suitable control of
the movement of small squads of robots in order to perform
unmanned vehicle operations, for instance \cite{pegoel}. Even,
these ideas have been proposed to model crowd motion, including
more ``intelligent'' particles deciding their movement based on
optimization of certain quantities: time to exit from a room or a
stadium, for instance \cite{BMP}.

Either in social or in biological sciences, these models encounter
many interesting features such as the spontaneous formation of
different pattern behaviors. When we talk about patterns, we do
not mean static patterns like in the study of crystals but rather
dynamic patterns leading to the collective motion of the
individual ensemble. For instance, two of the main collective
motion patterns studied in different models are the flock and the
milling behavior, see \cite{DCBC,CDP,review2,CKMT,CPM}. In the
flock pattern, individuals achieve a consensus on the direction or
orientation towards some objective, producing as a consequence a
particular spatial shape showing their preferred comfort
structure. This kind of swiftly moving flocks have been reported
in many species although the most spectacular or bucolic ones are
the bird flocks, starlings for instance. In the mill pattern,
individuals arrange into a kind of vortex like motion around some
point. This particular moving pattern has been observed in fish
schools. Hundreds of movies can be easily accessed through
internet search showing them.

There are many reasons one can argue, why such a large number of
individuals react to external stimuli producing these macroscopic
patterns without seemingly the presence of a leader in the swarm.
Hydrodynamic enhancement, predators avoidance, social
interactions, spawning survival rate, and many others have been
proposed to explain this behavior in different species, see
\cite{parrish}.

One of the main question in describing this behavior by
mathematical models is how to include the interaction between
individuals. In any case, there is a consensus that the modelling
starts from particle-like models as in statistical physics. These
particle models are also called Individual Based Models (IBMs) in
the community. They are usually formed by a set of differential
equations of Newton type (called 2nd order models) or by kinematic
equations where the inertia terms are neglected (called first
order models). Essentially, by admitting that the inertia term is
negligible, we assume that individuals can adjust to the velocity
field instantaneously, an approximation valid when their speed is
not too large. In any case, these first order models were proposed
in the literature derived in a phenomenological manner
\cite{mogilner,Mogilner2,parrish,TB04,Top06,EVL}. The literature
on first and second order models for swarming has increased
exponentially fast in the last few years. Many of these models
find also their origin in social sciences, where consensus or
opinion formation was also described in similar grounds. Another
typical ingredient in these models is some kind of noise leading
to systems of SDEs. In this work, we will not discuss how to
incorporate noise in these models, we refer to \cite{BCC} and the
references therein.

Most of these models are based on discrete approaches
incorporating certain effects that we like to call the ``first
principles'' of swarming. These first principles are based on
modelling the ``sociological behavior'' of animals with very
simple rules such as the social tendency to produce grouping
(attraction/aggregation), the inherent minimal space they need to
move without problems and feel comfortably inside the group
(repulsion/collisional avoidance) and the mimetic adaptation or
synchronization to a group (orientation/alignment). Even if these
minimal models contain very basic rules, the patterns observed in
their simulation and their complex asymptotic behavior are already
very challenging from the mathematical viewpoint. The 3-zone
models including attraction, repulsion, and alignment effects are
classical in fish modelling \cite{Ao,HW} for instance. Based on
them, one can incorporate may other effects to render more
realistic the outputs of the simulations and the models, see
\cite{BTTYB} for fish schools or \cite{HH} for birds flocks. We
also refer to the reader to the recent review \cite{review} about
the kinetic modelling of swarming.

To the eyes of a kinetic theorist or a statistical physicist,
studying such systems of ODEs when the number of individuals get
large is doomed to failure. Dynamical system approaches are quite
useful but they typically have huge problems to describe large
systems of particles. A classical approach to attack the problem
is to pass to a continuous description of the system. This means
to go from particle descriptions to kinetic descriptions where the
unknown is the particle density distribution in position-velocity
(phase) space for 2nd order models or in position space for 1st
order models.

Going from particle to continuum descriptions is one of the most
classical problems in kinetic theory. It is at the basis of the
derivation of the mother and father kinetic equations, namely: the
Vlasov and the Boltzmann equations. A rigorous derivation of the
Boltzmann equation from the Newtonian dynamics has only been given
for short times (of the order of the average time of first
collision), see \cite{Lan} \cite{GST}. In that case, interactions
between the particles are modelled by short-range potentials
leading to collision kernels. The question of the derivation of
the Boltzmann equation from particles with jump processes was also
raised and solved by \cite{Kac}, and further results are given in
the recent important work by \cite{M-M}. The derivation of the
Vlasov equation is well understood only for regular or not too
singular potentials \cite{BH,neunzert2,dobru,Hau-Ja}. In fact, a
full derivation of the Vlasov-Poisson system in 3D is also
lacking. The problem of passing to the limit from particle to
continuum models like the Vlasov equation is called the mean-field
limit. This name just comes from the fact that the resulting
equation is a kind of averaged version of the interaction between
the large number of individuals. Moreover, the resulting equation
gives the typical behavior of one isolated individual among all
the others since they are assumed to be completely
indistinguishable.

Finally, there are other famous mean-field limit equations, such
as the Euler and the Navier-Stokes equations for incompressible
fluids, see \cite{MP,MB}. It has been extensively used for
numerical purposes that both equations in the 2D incompressible
case can be derived from particle approximations, called vortex
point approximations. The convergence in the viscous case has been
rigorously proved for very general initial data \cite{Os,FHM}. In
the non-viscous case \cite{Sch} proves that particle
approximations converge towards solutions of the Euler equation,
but they may not converge to the good solution because of the lack
of uniqueness in the Euler equation, see \cite{DLS}. However, in
the case where the initial particles are equally spaced on a grid
to approximate a smooth solution of the Euler equation, the
convergence was shown in \cite{GHL}. These vortex methods have
been proven to be convergent and estimates of the error committed
have been obtained in recent works using optimal transport
techniques \cite{Hau} but not for the real Euler equation in 2D.

The aim of this work is to show in detail a particular example of
the mean field limit in the case of first order models not covered
in the previous literature. Nevertheless, we will first discuss
some of these issues for 2nd order models summarizing results in
\cite{CCR,BCC}. We will also discuss that the spatial shape of the
main patterns: flock and mills, are given by stationary solutions
of the 1st order models. This gives another reason from a more
conceptual mathematical viewpoint of reducing to 1st order models.
Section \ref{sec-mean-field} will be devoted to obtain the mean
field limit to the so-called aggregation equation for singular
potentials recovering some of the models studied in
\cite{B-C-L,B-L-R}. Here, the idea is to assume that we have
solutions of the model in better functional spaces due to the
singularity of the potential, but we have to pay in terms of
conditions on the initial distribution of particles (how they are
distributed) in such a way that the particle solution converges to
the continuum solution of the aggregation equation as
$N\to\infty$. We will make use of similar arguments to \cite{Hau}
to show the mean-field limit for first order swarming models with
singular potentials up to the Newtonian singularity. In Section
\ref{sec-exist}, we study a local existence of a unique
$L^p$-solution for the aggregation equation. This complements the
well-posedness theory in \cite{B-L-R}. Finally, Section
\ref{sec-propa} is devoted to show the propagation of chaos
property for the aggregation equation. This property is very
important from the physical relevance of the kinetic and
aggregation models, since it states that one can derive the
mean-field equations under quite generic randomly generated
initial location of the particles. We are only able to show it for
a more restricted set of singular potentials with respect to the
mean-field limit.

%
%
%
%
\section{The Dobrushin approach}
\subsection{Some Individual Based Models} As we described in
the introduction, the modelling in swarming starts by introducing
some particle models, IBMs in the jargon of this community,
incorporating some of the basic effects: repulsion, attraction,
and alignment. Let us discuss briefly some of these models,
starting with the ones that have recently attracted more attention
due to their simplicity while having a rich mathematical structure
and pattern formation. One of these models was introduced by the
UCLA group in \cite{DCBC} and it consists in Newton's like
equations where all the effect of repulsion and attraction is
encoded via a pairwise potential $W:\R^d\to \R$. A
popular choice for the interaction potential $W$ is the Morse
potential given by
\begin{equation}\label{morse-po}
W(x)= -C_A e^{-|x| /\ell_A} +
C_R e^{-|x| /\ell_R},
\end{equation}
where $C_A, C_R$ and $\ell_A, \ell_R$ are the strengths and the
typical lengths of attraction and repulsion, respectively. They
are chosen for having biologically reasonable potentials with
$C=C_R/C_A>1$ and $\ell_R/\ell_A<1$, see \cite{CPM} for other nice
choices of the interaction potentials and a deeper discussion on
the issue of biologically relevant interaction potentials. Apart
from this, the other effect included is the tendency of the
particles to travel asymptotically at a fixed speed as in
\cite{LR}. Consequently, a term producing a balance between
self-propulsion and friction is introduced imposing an asymptotic
speed to the particles (if other effects are ignored), but it does
not influence the orientation vector. The resulting ODE system
reads as:
\begin{equation*}
  \left\lbrace
    \begin{array}{ll}
      \displaystyle \frac{dx_i}{dt} = {v}_i,
      &\qquad (i = 1,\dots,N),
      \vspace{.3cm}
      \\
      \displaystyle \frac{dv_i}{dt} = (\alpha - \beta \,|{v}_i|^2) {v}_i
      - \frac1N \sum_{j \neq i } \nabla W (|x_i - x_j|),
      &\qquad (i = 1,\dots,N).
    \end{array}
  \right.
\end{equation*}
where $\alpha$, $\beta$ are nonnegative parameters, determining
the asymptotic speed of particles given by $\sqrt{\alpha/\beta}$.
Here, the potential has been scaled depending on the mass of each
particle as in \cite{CDP} and in such a way that the effect of the
potential per particle diminishes while the energy is of constant
order as the number of particles $N$ diverges. This scaling is the
so-called mean-field scaling, see the introduction of \cite{BV}
for a nice discussion of the different scalings in first order
models.

Another popular IBM including only the alignment effect is the
so-called \cite{CS2} model. Each individual in the swarm changes
its velocity vector based on the other individuals by
adjusting/averaging their relative velocity with all the others.
This averaging is weighted in such a way that closer individuals
have more influence than further ones. For a system with $N$
individuals the Cucker-Smale model reads as
$$
\left\{\begin{array}{lr} \displaystyle\frac{dx_i}{dt} = v_i, \\[3mm]
\displaystyle\frac{dv_i}{dt} = \frac1N \displaystyle\sum_{j=1}^{N}
w_{ij} \left(v_j - v_i\right) ,
\end{array}\right.
$$
with the \emph{communication rate} $w(x)$ given by:
$$
w_{ij} = w(x_i-x_j)= \frac 1 {\left(1 +|x_i - x_j|^2
  \right)^\gamma}\,,
$$
for some $\gamma \geq 0$.

Associated to the above models, one can formally write the
expected Vlasov-like kinetic equations as $N\to\infty$, see for
instance \cite{CDP}, leading to
\begin{equation}
  \label{eq:swarming}
  \partial_t f + v \cdot \grad_x f
  - (\grad W * \rho) \cdot \grad_v f
  + \dv_v((\alpha - \beta |v|^2) v f)
  = 0,
\end{equation}
where $\rho$ represents the macroscopic \emph{density} of $f$:
\begin{equation*}
  \rho(t,x) := \int_{\R^d} f(t,x,v) \,dv
  \quad \text{ for } t \geq 0, \quad x \in \R^d.
\end{equation*}
The Cucker-Smale particle model leads to the following kinetic
equation:
\begin{equation}\label{eq:CS}
\frac{\partial f}{\partial t} + v\cdot \nabla_x f = \nabla_v \cdot
\left[\xi[f] \,f\right],
\end{equation}
where $\xi[f](x,v,t) = \left( H \ast f \right)(x,v,t)$, with
$H(x,v)=w(x)v$ and $\ast$ standing for the convolution in both
position and velocity ($x$ and $v$). We refer to
\cite{CS2,HT08,HL08,cfrt} for further discussion about this model
and qualitative properties.

Moreover, quite general models incorporating the three effects
previously discussed with additional ingredients, such as vision
cones or topological interactions, have been considered in
\cite{review,LLE,Agueh,Albi,Hasko}. In particular in \cite{LLE},
they consider that the $N$ individuals follow the system:
\begin{equation}\label{eq:lle}
\left\lbrace
\begin{array}{l}
\displaystyle \frac{dx_i}{dt} = {v}_i, \vspace{.3cm}\\
\displaystyle \frac{dv_i}{dt} = F^A_i + F^I_i,
\end{array}
\right.
\end{equation}
where $F^A_i$ is the self-propulsion generated by the
$i$th-individual, while $F^I_i$ is due to interaction with the
others. The interaction with other individuals can be generally
modeled as:
$$
F^I_i=F^{I,x}_i + F^{I,v}_i = \sum_{j=1}^N
g_\pm(|x_i-x_j|)\frac{x_j-x_i}{|x_i-x_j|} + \sum_{j=1}^N
h_\pm(|v_i-v_j|)\frac{v_j-v_i}{|v_i-v_j|} .
$$
Here, $g_+$ and $h_+$ ($g_-$ and $h_-$) are chosen when the
influence comes from the front (behind), i.e., if $(x_j-x_i)\cdot
v_i> 0$ ($<0$); choosing $g_+ \neq g_-$ and $h_+ \neq h_-$ means
that the forces from particles in front and those from particles
behind are different. The sign of the functions $g_\pm(r)$ encodes
the short-range repulsion and long-range attraction for particles
in front of (+) and behind (-) the $i$th-particle. Similarly,
$h_+>0$ ($<0$) implies that the velocity-dependent force makes the
velocity of particle $i$ get closer to (away from) that of
particle $j$.

Some of these models, for instance \cite{Agueh,Albi,Hasko},
include sharp boundaries for the vision cone or for the
interaction with the nearest neighbors. As we shall see later,
these are typical situations in which the mean-field limit for
general measures will not work. By sharp boundaries we mean that
the functions involved in the kernels such as $w(x)$, $g_\pm$, or
$h_\pm$ are given by characteristic functions on sets depending on
the location/velocity of the agent.

\subsection{Basic tools in transport distances} In this
subsection, we present several definitions of Wasserstein
distances and their properties.
\begin{definition}(Wasserstein p-distance) \label{defdp}
Let $\rho_1,~ \rho_2$ be two Borel probability measures on
$\bbr^d$. Then the Euclidean Wasserstein distance of order $1\leq
p<\infty$ between $\rho_1$ and $\rho_2$ is defined as
\begin{equation*}
d_p(\rho_1,\rho_2) := \inf_{\gamma} \left( \int_{\bbr^d \times
\bbr^d} |x-y|^p \, d\gamma(x,y) \right)^{1/p},
\end{equation*}
and, for $p=\infty$ (this is the limiting case, as $p \to
\infty$),
\begin{equation*}
d_{\infty}(\rho_1,\rho_2) := \inf_{\gamma} \left( \gamma -
\sup_{(x,y) \in \bbr^d \times \bbr^d} | x- y | \right),
\end{equation*}
where the infimum runs over all transference plans, i.e., all
probability measures $\gamma$ on $\bbr^d \times \bbr^d$ with
marginals $\rho_1$ and $\rho_2$ respectively,
\[
\int_{\bbr^d \times \bbr^d} \phi(x) d\gamma(x,y) = \int_{\bbr^d}
\phi(x) \rho_1(x) dx,
\]
and
\[
\int_{\bbr^d \times \bbr^d} \phi(y) d\gamma(x,y) = \int_{\bbr^d}
\phi(y) \rho_2(y) dy,
\]
for all $\phi \in \mathcal{C}_b(\bbr^d)$.
\end{definition}

We also remind the definition of the push-forward of a measure by
a mapping in order to give the relation between Wasserstein
distances and optimal transportation.

\begin{definition}
Let $\rho_1$ be a Borel measure on $\bbr^d$ and $\mathcal{T} :
\bbr^d \to \bbr^d$ be a measurable mapping. Then the push-forward
of $\rho_1$ by $\mathcal{T}$ is the measure $\rho_2$ defined by
\[
\rho_2(B) = \rho_1(\mathcal{T}^{-1}(B)) \quad \mbox{for} \quad B
\subset \bbr^d,
\]
and denoted as $\rho_2 = \mathcal{T} \# \rho_1$.
\end{definition}

The set of probability measures with bounded moments of order $p$,
denoted by $\mathcal{P}_p(\bbr^d)$, $1\leq p< \infty$, is a
complete metric space endowed with the $p$-Wassertein distance
$d_p$, see \cite{Vil}. We refer to \cite{GS,MR2219334} for more
details in the case of the $d_\infty$ distance.

\begin{remark}
The definition of $\rho_2 = \mathcal{T} \# \rho_1$ is equivalent
to
$$
\int_{\bbr^d} \phi(x)\, d\rho_2(x) = \int_{\bbr^d} \phi(\mathcal
T(x))\, d\rho_1(x)\,,
$$
for all $\phi\in \mathcal{C}_b(\R^d)$. Given a probability measure with
bounded $p$-th moment $\rho_0$, consider two measurable mappings
$X_1,X_2 : \bbr^d \to \bbr^d$, then the following inequality
holds.
\[
d_p^p(X_1 \# \rho_0, X_2 \# \rho_0) \leq \int_{\bbr^d \times
\bbr^d} |x-y|^p d\gamma(x,y) = \int_{\bbr^d} | X_1(x) - X_2(x)|^p
d\rho_0(x).
\]
Here, we used as transference plan $\gamma = (X_1 \times
X_2)\#\rho_0$ in Definition \ref{defdp}.
\end{remark}

\subsection{A quick review of the classical Dobrushin result}

Under smoothness assumptions on the ingredient functions of the
swarming models, one can use adaptations of the classical result
of \cite{dobru} to obtain what is called the mean-field limit
equation for general particle approximations of any initial
measure. These arguments are classical in kinetic theory and were
also introduced in \cite{BH,neunzert2}, making use of the bounded
Lipschitz distance, and reviewed in
\cite{spohn2,villani-champmoyen}, see also \cite{Szn,Mel96} for
the case with noise. The bounded Lipschitz distance or dual
$W^{1,\infty}$-norm is equivalent to the Wasserstein distance
$d_1$ for compactly supported measures. This strategy works as
soon as the velocity field defining the characteristics of the
model is a bounded and globally Lipschitz function whose
dependence on the measure itself is Lipschitz continuous in the
$d_1$ sense. These ideas were improved to allow for locally
Lipschitz velocity fields for compactly supported initial measures
in \cite{CCR} and for suitable decay conditions at infinity and
with noise in \cite{BCC}. With these techniques one can include
quite general kinetic models for swarming in this well-posedness
theory.

Let us introduce some notation for this section: $\mathcal A =
\Pc(\R^d\times\R^d)$ denotes the subset of
$\mathcal{P}(\R^d\times\R^d)$ consisting of measures of compact
support in $\R^d\times\R^d$.  On the other hand, we consider the
set of functions $\mathcal{ B}:=\lip_{loc} (\R^d\times\R^d)$,
which in particular are locally Lipschitz with respect to $(x,v)$.
$B_R$ will denote the ball centered at 0 of radius $R$ in $\R
\times \R$.

Let us consider general operators from  measures to  vector
fields, $\mathcal{ H}[\cdot]:\mathcal{A}\to \mathcal{B}$,
satisfying the following hypotheses: for any $R_0 > 0$ and $f, g
\in \mathcal{A}$ such that $\supp f \cup \supp g \subseteq
B_{R_0}$, there exists some ball $B_R \subset \R^d \times \R^d$
and a constant $C = C(R,R_0)>0$, such that
  \begin{gather}\label{hyp:general-model}
    \|\mathcal{H}[f]-\mathcal{H}[g]\|_{L^\infty(B_R)}
    \leq  C \, d_1(f,g),
    \\
    \lip_R(\mathcal{H}[f])
    \leq C, \qquad \|\mathcal{H}[f]\|_{L^\infty(B_R)} \le C
.\label{hyp:general-model2}
  \end{gather}
Here, $\lip_R(\cdot)$ denotes the Lipschitz constant of a function
in $B_R$.

Given $f \in \mathcal{C}([0,T], \Pc(B_{R_0}))$,
and for any initial condition $(X^0, V^0) \in \R^d \times \R^d$,
the following system of ordinary differential equations has a
unique locally defined solution thanks to
conditions~\eqref{hyp:general-model}
\begin{subequations}
  \begin{align}
    \label{eq:characteristicsX-general}
    \frac{d}{dt} X &= V,   & X(0) = X^0
    \\
    \label{eq:characteristicsV-general}
    \frac{d}{dt} V &= \mathcal{H}[f(t)](X,V) ,   & V(0) = V^0.
  \end{align}
\end{subequations}
We will additionally require that the solutions to that system are global.
Of course, this is a requirement that has to be checked for every
particular model. We prefer to give a general condition which
reduces the problem of existence and stability to the simpler one
of existence of the ODEs. Under the above conditions, the
 existence and uniqueness of associated transport equation
 \begin{equation}
  \label{eq:general-model}
  \partial_t f + v \cdot \grad_x f
  - \nabla_v \cdot [ \mathcal{H}[f] f ]
  = 0.
\end{equation}
 was obtained in \cite{CCR} to which we refer for full details. In \cite{CCR}, the interactions $\mathcal{
H}[f]=(\alpha-\beta|v|^2)v-\nabla W * \rho$ and $\mathcal{H}[f]=H
*f$ corresponding to \eqref{eq:swarming} and \eqref{eq:CS},
respectively, and
$$
  \mathcal{H}[f]=F_A(x,v)+G(x) * \rho + H(x,v) * f,
$$
with $F_A$, $G$ and $H$ given functions satisfying suitable
hypotheses, such that the kinetic equation
\eqref{eq:general-model} corresponds to the model \eqref{eq:lle} are investigated.
\medskip

\begin{theorem}
  \label{thm:Existence-general}
  Given an operator $\mathcal{H}[\cdot]:\mathcal{A}\to \mathcal{B}$
  satisfying Hypotheses \eqref{hyp:general-model} and
  \eqref{hyp:general-model2} for which the characteristics \eqref{eq:characteristicsX-general}-\eqref{eq:characteristicsV-general}
 are globally well-defined, and $f_0$ a measure on $\R^d \times \R^d$
  with compact support. There exists a solution $f$ on $[0,+\infty)$
  to equation \eqref{eq:general-model} with initial condition
  $f_0$. In addition,
  \begin{equation}
    \label{eq:f-continuous-general}
    f \in \mathcal{C}([0,+\infty); \Pc(\R^d\times \R^d))
  \end{equation}
  and there is some increasing function $R = R(T)$ such that for all $T>0$,
  \begin{equation}
    \label{eq:supp-f-general}
    \supp f_t \subseteq B_{R(T)} \subseteq \R^d \times \R^d
    \quad \text{ for all } t \in [0,T].
  \end{equation}
  This solution is unique among the family of solutions satisfying
  \eqref{eq:f-continuous-general} and \eqref{eq:supp-f-general}.
  Moreover, given any other initial data $g_0\in \Pc(\R^d\times\R^d)$ and $g$
  its corresponding solution, then there exists a
  strictly increasing function $r(t):[0,\infty)\to \R^+_0$
  with $r(0)=1$ depending only on $\mathcal{H}$ and the size of the support of
  $f_0$ and $g_0$, such that
  \begin{equation*}
    d_1(f_t, g_t)
    \leq
    r(t)\, d_1(f_0, g_0),
    \quad
    t \geq 0.
  \end{equation*}
\end{theorem}

\

The stability theorem \ref{thm:Existence-general} gives in
particular a rigorous derivation of the kinetic equation
\eqref{eq:general-model} from the large particle limit of the
system of ordinary differential equations. This is the exact
statement of the mean-field limit for general measures as initial
data. Let us consider the system of ordinary differential
equations:
  \begin{subequations}
    \label{eq:swarming-odes}
    \begin{align}
      &\dot{x_i} = v_i,
      \quad &i=1,\dots,N,
      \\
      &\dot{v_i} = \sum_{j \neq i} m_j \mathcal{H}[f^N(t)](x_i,v_i),
      \quad &i=1,\dots,N.
    \end{align}
  \end{subequations}
where $m_1,\dots,m_N\geq 0$ and $\sum_i m_i=1$ and $f^N$ is
defined next. Under the conditions of Theorem
\ref{thm:Existence-general}, we first notice that if
$x_i,v_i:[0,T] \to \R^d$, for $i = 1,\dots,N$, are a solution to
the system \eqref{eq:swarming-odes}, then the function $f^N:[0,T]
\to \Pc(\R^d\times\R^d)$ given by
  \begin{equation}
    \label{eq:ode-equivalent-measure}
    f^N_t := \sum_{i=1}^N m_i\, \delta_{(x_i(t), v_i(t))}
  \end{equation}
is the solution to \eqref{eq:general-model} with initial condition
  \begin{equation}
    \label{eq:ode-equivalent-measure-initial}
    f^N_0 = \sum_{i=1}^N m_i\, \delta_{(x_i(0), v_i(0))}.
  \end{equation}
In fact, the solution \eqref{eq:ode-equivalent-measure} is called
the empirical measure associated to the system of ODEs
\eqref{eq:swarming-odes}. We finally write the full statement of
the mean-field limit in the Dobrushin strategy.

\begin{corollary}  \label{cor:N-particle}
  Given $f_0\in \Pc(\R^d\times\R^d)$ and $\mathcal{H}[f]$ satisfying the conditions of Theorem
  \ref{thm:Existence-general}, take a sequence of $f_0^N$ of measures
  of the form \eqref{eq:ode-equivalent-measure-initial} (with $m_i$,
  $x_i(0)$ and $v_i(0)$ possibly varying with $N$), in such a way that
  $$
  \lim_{N\to \infty} d_1(f_0^N, f_0) = 0.
  $$
  Consider $f^N_t$ the empirical measure associated to the solution of the system
  \eqref{eq:swarming-odes} with initial conditions $x_i(0)$,
  $v_i(0)$. Then,
  \begin{equation}\label{meanfield}
  \lim_{N\to \infty} d_1(f_t^N, f_t) = 0 ,
  \end{equation}
  for all $t\geq 0$, where $f = f(t,x,v)$ is the unique measure
  solution to eq. \eqref{eq:general-model} with initial data $f_0$.
\end{corollary}

This section can be directly applied to the models recently
introduced in \cite{Agueh} to account for vision cones and
braking/acceleration of individuals and those in \cite{Albi,Hasko}
to include topological (nearest neighbours) interactions once the
parameter functions are smoothed out to avoid sharp boundaries.

\subsection{First-order models: Aggregation Equation}

Summarizing the previous subsection, under suitable smoothness of
the parameters involved in the swarming models, the empirical
measures are solutions themselves of the Vlasov-like kinetic
equation \eqref{eq:general-model}. Thus, an stability result in
$d_1$ with respect to the initial data is enough to conclude the
mean-field limit. Let us consider one of the particular examples
in subsection 2.1, the model introduced in \cite{DCBC} with the
Morse potential \eqref{morse-po}. This potential does not satisfy the smoothness assumption in Theorem \ref{thm:Existence-general}. In principle, one cannot
expect to have a mean-field result for general measures as initial
data and for general approximations by particles. In fact, we do
not have a well-posedness theory for such initial data in those
cases. However, one can develop well-posedness theories in better
functional spaces, say $L^1\cap L^p(\R^d\times\R^d)$ for the
initial data and then impose suitable conditions to the
distribution of the approximated particles initially to be able to
conclude the mean-field limit \eqref{meanfield}. This is the
strategy that have been followed in \cite{Hau-Ja} for the
classical Vlasov equation and in \cite{Hau} for Euler-like
equations in fluid mechanics.

In the next sections, the objective is to show this strategy
applied to a simpler swarming model than the ones showed above. We
will showcase these tools in the case of the so-called aggregation
equation. Let us assume that we have just particles interacting
through the pairwise potential $W(x)$. Assuming that the
variations of the velocity and speed are much smaller than spatial
variations, see \cite{Mogilner2}, then one can neglect the inertia
term in Newton's equation to deduce that
$$
\frac{d X_i}{dt} = - \sum_{j \neq i } \nabla W (X_i - X_j) \mbox{
in the $N\to\infty$ limit} \Rrightarrow
\begin{cases}
\displaystyle\frac{\partial\rho}{\partial t}+\text{div} \left(\rho u\right)=0  \\
u=-\nabla W*\rho
\end{cases}.
$$
Another reason to study this first order equation is that the
stationary states of the first order model determine the spatial
shape of the flock solutions to the second order models, see
\cite{CPM}. Let us remark that one could apply the Dobrushin
strategy to the aggregation equation for $\mathcal{C}^2(\R^d)$ smooth
potential with at most quadratic growth at infinity by following
the same argument as in Theorem \eqref{thm:Existence-general}.
This argument was detailed in a nice summer school notes in
\cite{Gol}.

%
%
%
%

\section{Mean-Field Limit for the Aggregation
Equation}\label{sec-mean-field}

Now, we analyse the mean-field limit of the first order model for
swarming introduced in the previous section. More precisely, we
will study sufficient conditions on the initial distribution of
particles for the convergence of a particle system towards the
aggregation equation. This model consists of the continuity
equation for the probability density of individuals $\rho(x,t)$ at
position $x \in \bbr^d$ and time $t
>0$ given by:
\begin{equation}\label{conti-agg}
\left\{ \begin{array}{ll}
\partial_t \rho + \nabla \cdot (\rho u) = 0, & \qquad t > 0, \quad x \in
\bbr^d,\\[2mm]
u(t,x) := - \nabla W * \rho, & \qquad t > 0, \quad x \in
\bbr^d,\\[2mm]
\rho(0,x) := \rho^0(x), & \qquad x \in \bbr^d,
\end{array} \right.
\end{equation}
where $u(x,t)$ is velocity field non-locally computed in terms of
the density of individuals.

As an approximation by particles of the aggregation equation
\eqref{conti-agg}, we consider the following ODE system:
\begin{equation}\label{disconti-agg}
\left\{ \begin{array}{ll}
\displaystyle \dot{X}_i(t) = -\sum_{j \neq i}m_j \nabla W (X_i(t) - X_j(t)), & \\
X_i(0)= X_i^0, \qquad i =1,\dots,N. &
\end{array} \right.
\end{equation}
Here, $\{X_i\}_{i=1}^{N}$ and $\{m_i\}_{i=1}^{N}$ are the
positions and weights of $i$-th particles, respectively. We define the
associated empirical distribution $\mu_N(t)$ as
\begin{equation}\label{dirac-sol-dis}
\mu_N(t) = \sum_{i=1}^{N}m_i \delta_{X_i(t)}, \quad \sum_{i=1}^N
m_i = \int_{\bbr^d} \rho_0(x) dx=1,
\end{equation}
with $m_i > 0$, $i=1,\dots,N$. As long as two particles (or more)
do not collide, and if we set $\nabla W (0)= 0$ (arbitrarily if
there is a singularity), then $\mu_N$ satisfies \eqref{conti-agg}
in the sense of distributions, i.e., $\mu_N(t)$ and $\rho(t)$
satisfy the same equation. In this framework, the convergence:
\begin{center}
\textit{``$\mu_N^0 \rightharpoonup \rho^0$ weakly-$*$ as measures
$\Longrightarrow$ $\mu_N(t) \rightharpoonup \rho(t)$ weakly-$*$ as measures\\
for small time or for every time?''}
\end{center}
is a natural question. If the answer is yes, we say that the
continuity equation \eqref{conti-agg} is the mean-field limit of
the particle approximation \eqref{disconti-agg}. In other words,
we can say that the continuum nonlocal equation \eqref{conti-agg}
has been rigorously derived from particle systems.

Because of the singularity in the interaction force, the natural
transport distance to use is the one induced by the
$d_\infty$-topology. Remark that this distance also allows to
understand linearized stability of particle systems around
singular steady state measures with a ring shape in first order
aggregation models, see \cite{BCLR2,KSUB}. Actually, a local
perturbation of the dynamical system \eqref{disconti-agg} keeping
the number of particles fixed is obtained by transporting the
particle to other locations nearby. One could even allow for
splitting of the mass into different particles, but all of them
located in a local neighborhood of the unperturbed particle
positions. Certainly, sending a small portion of mass very far
away from the location of one particle is not a
$d_\infty$-perturbation of the atomic measure but it is a $d_p$
small perturbation for all $1\leq p<\infty$. These ideas have also
recently been used in \cite{BCLR2} to study local minimizers of
the energy functional associated to \eqref{conti-agg}.

Another issue to cope with is that we are dealing with particle
systems whose characteristics may lead to collisions in finite
time. Therefore, we will be able to obtain meaningful results only
on intervals in which collisions are avoided (although in some
particular cases we can allow collisions).

We next introduce several notations that are used throughout the
rest of this work to compare the distance between a solution
$\rho(t)$ of the continuum aggregation equation \eqref{conti-agg}
and the empirical measure $\mu_N(t)$ defined
by~\eqref{dirac-sol-dis} associated to a solution
$\{X_i\}_{i=1}^N$ of the particle system~\eqref{disconti-agg}. The
main two quantities appearing in this comparison are the
$d_\infty$-distance between $\rho(t)$ and $\mu_N(t)$, and the
minimum inter-particle distance:
\begin{equation} \label{notat-simp}
\eta(t) := d_{\infty}(\mu_N(t),\rho(t)), \quad \eta_m(t) :=
\min_{1 \leq i \neq j \leq N} \left(|X_i(t) - X_j(t)|\right),
\end{equation}
with $\eta^0 := \eta(0)$ and $\eta_m^0 := \eta_m(0)$. Our strategy
does not take advantage, as we do not know how, of the repulsive
or attractive character of the potentials, being the proof equal
for both cases.

A theory of well-posedness for measure solutions have been
obtained for the aggregation equation \eqref{conti-agg} allowing
collision of particles in finite time in \cite{CDFLS,CDFLS2}. In
these works, the potential is assumed to be smooth except at the
origin, where the allowed singularity cannot be worse that
Lipschitz and the potential has to be $\lambda$-convex, see
\cite{CDFLS} for details. This convexity allows for attractive at
the origin potentials, but not repulsive, with local behaviors of
the form $|x|^b$ with $1\leq b<2$. In these works, the essential
tools that allow to get the mean-field limit for more singular
potentials that quadratic are based on gradient flows in the
Wasserstein distance $d_2$ sense as in \cite{AGS}. The additional
dissipation in the system of the natural Liapunov functional given
by the total interaction energy is crucial to get the mean field
limit for general measures for a potential behaving locally at $0$
like $W(x)\simeq |x|$, for instance for the attractive Morse
potential $W(x)=1-e^{-|x|}$.

In this work, we want to allow for more singular potentials at the
origin as in \cite{B-C-L,B-L-R}, and thus we need to work with
solutions in better functional spaces. More precisely, we will
work with solutions of the aggregation equation \eqref{conti-agg}
in $L^{\infty}(0,T;(L^{1} \cap L^{p})(\bbr^d))$ with $1\leq p \leq
\infty$ to be determined depending on the singularity of the
potential. We will use the notation
\begin{equation*} 
\|\rho\|_{(L^1 \cap L^p)(\R^d)} := \|\rho\|_{1} + \|\rho\|_{p},
\quad \|\rho\| := \|\rho\|_{L^{\infty}(0,T;(L^{1} \cap
L^{p})(\R^d))}\,,
\end{equation*}
where $\|\rho\|_p$ denotes the $L^p(\bbr^d)$-norm of $\rho$,
$1\leq p\leq \infty$.

In order to make sense of solutions to \eqref{conti-agg}, we need
the following assumptions on the interaction potential: we first
fix $W(0) = 0$ by definition, even if $W$ is singular at the
origin, and
\begin{equation}\label{main-assum-1}
|\nabla W(x)|\leq \frac{C}{|x|^{\alpha}}, \quad \mbox{and} \quad
|D^2 W(x)| \leq \frac{C}{|x|^{1+\alpha}}, \quad \forall ~ x \in
\bbr^d \backslash \{ 0\} \,,
\end{equation}
for $-1\leq\alpha<d-1$. Note that due to the assumptions on $W$,
we can always find $1<p<\infty$ such that $(\alpha + 1)p^{\prime}
< d$, and thus $\nabla W$ belongs to
$\mathcal{W}^{1,p'}_{loc}(\bbr^d)$.

Our results also apply with minor modifications for interaction
potentials of the form $W := W_1 + W_2$, with $W_1$ satisfying
assumptions \eqref{main-assum-1}, and $\nabla W_2$ being a global
Lipschitz function, or even more general satisfying a one-sided
Lipschitz (or convexity) condition $y \cdot D^2 W_2(x) y \le C
|y|^2$ for all $y \in \bbr^d$. This last generalization is
important because it is satisfied if $W_2 = c |x|^a, (0\leq a \leq
2)$ with $c$ positive. So that any repulsive-attractive potential
$W$, see \cite{BCLR,BCLR2} for a definition, such that $W(x)\simeq
-|x|^b/b$ locally at $x$ near the origin, satisfies assumptions
\eqref{main-assum-1} locally with $\alpha=1-b$. Therefore, our
mean-field limit results apply to locally repulsive potentials
with exponent range $2-d<b<a\leq2$ and without much restriction on
the attractive part at $+\infty$, i.e., $a>0$. We will discuss
further on localizing assumptions \eqref{main-assum-1} at the end
of this section. Finally, we cannot apply our techniques to the
Newtonian singularity \cite{BLL} being the limiting case of our
strategy as it was the case for the Euler-like models in fluid
mechanics studied in \cite{Hau}.

We next summarize the results on the existence and uniqueness of
solutions to the aggregation equation \eqref{conti-agg}. For the
local well-posedness of solutions to equation \eqref{conti-agg},
we refer to \cite{BL,B-C-L,B-L-R,L}. In particular, unique
solutions for the system \eqref{conti-agg} were obtained in
\cite{B-L-R} with second moment bounded initial data. More
precisely, Bertozzi et al. \cite[Theorem 1.1]{B-L-R} showed that
if $\nabla W \in \mathcal{W}^{1,p^{\prime}}(\R^d)$ and $\rho^0 \in
L^p(\bbr^d) \cap \mathcal{P}_2(\bbr^d)$, then there exists $~T^* >
0$ and a unique nonnegative solution to~\eqref{conti-agg}
satisfying
\[
\rho \in \mathcal{C}([0,T^*],(L^{1} \cap L^{p})(\bbr^d)) \cap \mathcal{C}^1([0,T^*],
\mathcal{W}^{-1,p}(\bbr^d)).
\]
Unfortunately, one can not directly apply those results for
potentials satisfying assumptions \eqref{main-assum-1}. We will
compliment the results in \cite{B-L-R} to show the local existence of
a unique solution to the system \eqref{conti-agg} with the
interaction potential function $W$ satisfying \eqref{main-assum-1}
in Section \ref{sec-exist}. We prefer to postpone the
well-posedness theory in order to emphasize the mean-field limit
result contained in the following theorem, whose proof follows the
strategy in \cite{Hau}.

\begin{theorem}\label{main-thm}
Suppose the kernel $W$ satisfies \eqref{main-assum-1}, and let
$\rho$ be a solution to the system \eqref{conti-agg} up to time
$T>0$, such that $\rho\in L^{\infty}(0,T;(L^{1}\cap
L^{p})(\bbr^d))\cap \mathcal{C}([0,T],\mathcal{P}_1(\R^d))$, with
initial data $\rho^0\in (\mathcal{P}_{1} \cap L^{p})(\bbr^d)$,
$0\leq\alpha<-1+d/p'$, and $1<p\leq \infty$. Furthermore, we
assume $\mu_N^0$ converges to $\rho^0$ for the distance $d_\infty$
as the number of particles $N$ goes to infinity, i.e.,
\[
d_{\infty}(\mu_N^0,\rho^0) \to 0 \quad \mbox{as} \quad N \to \infty,
\]
and that the initial quantities $\eta^0,~\eta_m^0$ satisfy
\begin{equation}\label{ini-assum}
\lim_{N \to \infty} \frac{(\eta^0)^{d/p^{\prime}}}{(\eta^0_m)^{1 + \alpha}} = 0.
\end{equation}
Then, for $N$ large enough the particle system
\eqref{disconti-agg} is well-defined up to time $T$, in the sense
that there is no collision between particles before that time, and
moreover
\[
\mu_N(t) \rightharpoonup \rho(t) \quad \mbox{weakly-$*$ as
measures as} \quad N \to \infty, \quad \mbox{for all} \quad t \in
[0,T].
\]
\end{theorem}

\begin{remark}\label{rem:initialdata}
Let us first discuss the assumptions on the initial data in
Theorem \ref{main-thm}. The mean-field limit is valid for
particular approximations $\mu_N^0$ of $\rho^0$, that is, for well
chosen particle approximations of the initial data. In fact, a
procedure to construct initial atomic measures approximating the
initial condition in the sense of \eqref{ini-assum} is the
following: define a regular mesh of size $\ep$ and approximate
$\rho^0$ by a sum of Dirac masses $\mu_N^0$ located at the center
of the cells such that the mass at each particle is exactly equals
to the mass of $\rho^0$ contained in the associated cell. In that
case, we have $\eta^0 \sim \ep$ and $\eta_m^0 \sim \ep$ (for the
last condition we need that the mesh has some regularity). In that
case, the assumption~\eqref{ini-assum} is automatically fulfilled
since $(1+ \alpha) p' < d$. Notice that no bound on the masses
$m_i$ of the particles is required.
\end{remark}

\begin{proof}[Proof of Theorem \ref{main-thm}]
The proof of Theorem \ref{main-thm} is divided into three steps:
\begin{itemize}
\item In Step A, we estimate the growth of the $d_\infty$
Wasserstein distance between the continuum and the discrete
solutions $\eta$ that involves $\eta$ itself and $\eta_m$ in the
form:
    \begin{equation}\label{out-est-1}
    \frac{d\eta}{dt} \leq C \eta \|\rho\| \left( 1 + \eta^{d/p^{\prime}} \eta^{-(1 + \alpha)}_m \right).
    \end{equation}
\item In Step B, we estimate the decay of the minimum
inter-particle distance $\eta_m$, which also involves the terms
$\eta$ and $\eta_m$ in the form:
    \begin{equation}\label{out-est-2}
    \frac{d\eta_m}{dt} \geq - C \eta_m \|\rho\| \left( 1 + \eta^{d/p^{\prime}} \eta_m^{-(1 + \alpha)}\right).
    \end{equation}
\item In Step C, under the assumption of the initial approximation
\eqref{ini-assum}, we combine \eqref{out-est-1} and
\eqref{out-est-2} to conclude the desired result.
\end{itemize}

\

\textbf{Step A.-} We first introduce the flows generated by the
two velocity fields: $u(x,t) = -\nabla W * \rho$ and $u_N :=
-\nabla W
* \mu_N$. Let us remark that the convolution in the definition of
$u_N$ is just a notation for the right-hand side of
\eqref{disconti-agg} since the convolution of a Dirac Delta with a
(possibly) singular potential is not well-defined. These flows
$\Psi_N, \Psi:\bbr_{+} \times \bbr_{+} \times \bbr^d \to \bbr^d$
are defined as solutions of
\begin{equation}\label{traj}
\left\{ \begin{array}{ll}
\displaystyle \frac{d}{dt} (\Psi(t;s,x)) = u(t;s,\Psi(t;s,x)), &
\\[2mm]
\Psi(s;s,x) = x, &
\end{array} \right.
\end{equation}
for all $s,t \in [0,T]$, and
\begin{equation}\label{traj2}
\left\{  \begin{array}{ll}
\displaystyle \frac{d}{dt} (\Psi_N(t;s,x)) = u_N(t;s,\Psi_N(t;s,x)), &
\\[2mm]
\Psi_N(s;s,x) = x, &
\end{array} \right.
\end{equation}
for all $s,t\in [0,T^N_0]$. Notice that the solution $X_i(t)$ to
the system \eqref{disconti-agg} is well-defined and continuous by
the Cauchy-Lipschitz theorem as long as there is no collision
between particles. Since $\eta^0_m>0$, there exists $T^N_0
>0$ such that $\eta_m(t) > 0$ for $t \in [0,T^N_0]$ by continuity.
Then the flow map $\Psi_N(t;s,x)$ solution to \eqref{traj2} is
well-defined for $t,s \in [0,T^N_0]$. Now, let us check that the
flow for the solution associated to the continuum equation in
\eqref{traj} is well-defined. Assumptions \eqref{main-assum-1}
imply that
\begin{equation}\label{key}
|\nabla W (x) - \nabla W (y)| \leq \frac{2|x-y|}{\min
(|x|,|y|)^{\alpha + 1}}\,.
\end{equation}
One can see this by integrating along a straight line joining $x$
and $y$ but avoiding the singularity using a small circle if
needed, see \cite{Hau}. The estimate \eqref{key} implies that the
velocity field is Lipschitz continuous with respect to the spatial
variable. Actually, one can estimate it as
\begin{align*}
|u(t,x)-u(t,y)| &\leq \int_{\bbr^d} | \nabla W(x-z) - \nabla W
(y-z) | \rho(t,z) \,dz \\
&\leq 2|x-y| \int_{\bbr^d}\frac{1}{\min (|x-z|,|y-z|)^{\alpha +
1}} \rho(t,z) \,dz \\
&\leq 4|x-y| \sup_{x\in \bbr^d}\int_{\bbr^d}\frac{1}{|x-z|^{\alpha
+ 1}} \rho(t,z) \,dz\,.
\end{align*}
Now, splitting the last integral into the near- and far-field sets
$\mathcal{A} := \{ z : |x - z| \geq 1 \}$ and $\mathcal{B} :=
\bbr^d - \mathcal{A}$ and estimating the two terms, we deduce
\begin{align}
\int_{\bbr^d}\frac{1}{|x-z|^{\alpha + 1}} \rho(t,z) \,dz &\leq
\|\rho(t)\|_{1} + \left( \int_{\mathcal{B}} \frac{1}{|x-y|^{(1 +
\alpha)p^{\prime}}} dy \right)^{1/p^{\prime}}\|\rho(t)\|_{p}\nonumber\\
&\leq C \|\rho\|\,,\label{nearfar}
\end{align}
for all $x\in \bbr^d$ due to the assumption $(1 +
\alpha)p^{\prime}<d$. Putting together previous inequalities, we
get the desired Lipschitz continuity of the velocity field with
respect to $x$, which is moreover uniform in time. A similar
estimate using \eqref{main-assum-1} shows that the velocity field
is bounded, and then the flow $\Psi$ in \eqref{traj} is
well-defined. Our first aim is to find an expansion of the
velocity of the $d_\infty$ Wasserstein distance. The idea is
similar to the evolution of the euclidean Wassertein distance in
\cite{Carrillo-McCann-Villani03,Carrillo-McCann-Villani06,otto:geom:99}.
Fixed $0\leq t_0 < \min(T,T_0^N)$ and choose an optimal transport
map for $d_{\infty}$ denoted by $\mathcal{T}^0$ between
$\rho(t_0)$ and $\mu_N(t_0)$; $\mu_N(t_0) = \mathcal{T}^0 \#
\rho(t_0)$. It is known that such an optimal transport map exists
when $\rho(t_0)$ is absolutely continuous with respect to the
Lebesgue measure \cite{C-D-P}. Then it follows from Theorem
\ref{BLRlike-thm} that $\rho(t) = \Psi(t;t_0,\cdot ~ ) \#
\rho(t_0)$ and obviously $\mu_N(t) = \Psi_N(t;t_0,\cdot ~ ) \#
\mu_N(t_0)$ for $t\geq t_0$. We also notice that for $t\geq t_0$
\[
\mathcal{T}^t \# \rho(t) = \mu_N(t), \quad \mbox{where} \quad
\mathcal{T}^t = \Psi_N(t;t_0,\cdot) \circ \mathcal{T}^0 \circ
\Psi(t_0;t,\cdot).
\]
By Definition \ref{defdp} of the $d_p$ Wasserstein distance, we
get
\[
d_p^p\left( \mu_N(t), \rho(t) \right) \leq \int_{\bbr^d}
|\Psi(t;t_0,x) - \Psi_N(t;t_0,\mathcal{T}^0(x))|^p \rho(t_0,x) dx.
\]
In the case of $p = \infty$, we obtain
\[
\eta(t) = d_{\infty}(\mu_N(t),\rho(t)) \leq \| \Psi(t;t_0,\cdot) -
\Psi_N(t;t_0,\cdot)\circ \mathcal{T}^0 \|_{\infty}.
\]
We notice that
\[
\frac{d}{dt} \left( \Psi_N(t;t_0,\mathcal{T}^0(x)) - \Psi(t;t_0,x)
\right)\Big|_{t=t_0} =  u_N(t_0,\mathcal{T}^0(x)) - u(t_0,x).
\]
Thus, writing the integral form, dividing by $t-t_0$, and taking
the limit $t\to t_0^+$ we easily get
\begin{equation}\label{mt-est-1}
\frac{d}{dt} \| \Psi_N(t;t_0,\cdot) \circ \mathcal{T}^0 -
\Psi(t;t_0,\cdot) \|_{\infty} \Big|_{t=t_0^+} \leq \|
u_N(t_0,\cdot) \circ \mathcal{T}^0 - u(t_0,\cdot) \|_{\infty}.
\end{equation}
We now note that
\begin{align*}
u_N(t_0,\mathcal{T}^0 &(x)) - u(t_0,x)  \\
&=-\int_{\bbr^d} \nabla W(\mathcal{T}^0(x) - y)d\mu_N(t_0,y) +
\int_{\bbr^d} \nabla W (x-y) \rho(t_0,y) dy \cr  &= -
\int_{\bbr^d} \left( \nabla W(\mathcal{T}^0 (x) - \mathcal{T}^0
(y)) - \nabla W (x-y) \right)\rho(t_0,y) dy.
\end{align*}
For notational simplicity, we omit the time dependency on $t_0$ in
the next few computations. This yields that \eqref{mt-est-1} can
be rewritten as
\begin{equation}\label{mt-est-2}
\frac{d^+\eta}{dt} \leq C \sup_{x \in \bbr^d} \int_{\bbr^d} |
\nabla W(\mathcal{T} (x) - \mathcal{T} (y)) - \nabla W (x-y) |
\rho(y) dy .
\end{equation}
We decompose the integral on $\bbr^d$ into the near- and the
far-field parts as $\mathcal{A} := \{ z : |x - z| \geq 4\eta \}$ and $\mathcal{B} :=
\bbr^d - \mathcal{A}$ as
\begin{align}\label{mt-est-3}
\begin{aligned}
\int_{\bbr^d} | \nabla W(\mathcal{T} (x) - \mathcal{T} (y)) -
\nabla W (x-y) | \rho(y) dy &= \int_{\mathcal{A}} \cdots + \int_{\mathcal{B}} \cdots
\cr &:= \mathcal{I}_1 + \mathcal{I}_2.
\end{aligned}
\end{align}
For the estimate in the set $\mathcal{A}$, we use
\[
|\mathcal{T} (x) - \mathcal{T} (y)| \geq |x-y| - |\mathcal{T} (x)
- x| - |\mathcal{T} (y) - y| \geq |x-y| - 2\eta \geq \frac{|x-y|}2
\]
together with \eqref{key} and \eqref{nearfar} to obtain
\begin{align}
\mathcal{I}_1 &\leq \int_{\mathcal{A}} \frac{2\left( |x - \mathcal{T} (x)| +
|y - \mathcal{T} (y)| \right)}{\min (|x-y|,|\mathcal{T}
(x)-\mathcal{T} (y)|)^{\alpha + 1}} \rho(y) dy \nonumber\\ &\leq
4\eta \int_{\mathcal{A}} \left( \frac{1}{|x-y|^{\alpha + 1}} + \frac{2^{\alpha
+ 1}}{|x-y|^{\alpha + 1}} \right)\rho(y) dy \leq C\eta\int_{\mathcal{A}}
\frac{1}{|x-y|^{\alpha + 1}} \rho(y) dy \nonumber\\ &\leq C\eta
\int_{\bbr^d} \frac{1}{|x-y|^{\alpha + 1}} \rho(y) dy \leq C\eta
\|\rho\|.\label{mt-est-4}
\end{align}
For the second part $\mathcal{I}_2$, we estimate separately each
term using \eqref{main-assum-1} to deduce
\begin{align}\label{mt-est-5}
\begin{aligned}
\mathcal{I}_2 &\leq \int_{\mathcal{B}} \frac{\rho(y)}{|x-y|^{\alpha}} dy +
\int_{\mathcal{B}} \frac{\rho(y)}{\eta^{\alpha}_m}dy \cr &\leq \left( \int_{\mathcal{B}}
\frac{1}{|x - y|^{\alpha p^{\prime}}} dy \right)^{1/p^{\prime}}
\|\rho\|_{p} + \frac{1}{\eta_m^{\alpha}} \left(\int_{\mathcal{B}} 1 dy
\right)^{1/p^{\prime}} \|\rho\|_{p} \cr &\leq
C(\eta^{d/p^{\prime}-\alpha} + \eta^{d/p^{\prime}}
\eta^{-\alpha}_m ) \|\rho\|_{p} \leq C (\eta^{d/p^{\prime}-\alpha}
+ \eta^{d/p^{\prime}} \eta^{-\alpha}_m) \|\rho\|\,.
\end{aligned}
\end{align}
Notice that $|\mathcal{T} (x) - \mathcal{T} (y)|\geq \eta_m$ by
definition of the minimum inter-particle distance
\eqref{notat-simp} as soon as $\mathcal{T} (x) \neq \mathcal{T}
(y)$, $\nabla W(\mathcal{T} (x) - \mathcal{T} (y))=0$ otherwise.

Finally, we choose two indices $i,j$ so that $|X_i - X_j| =
\eta_m$, then we observe that the middle point between $X_i$ and
$X_j$ has to be transported by $\mathcal{T}$ to either $X_i$ or
$X_j$, and thus $\eta_m \leq 2\eta$. Hence by combining
\eqref{mt-est-2}-\eqref{mt-est-5} and being $t_0$ arbitrary in
$[0,\min(T,T_0^N))$, we have
\begin{equation}\label{ma-re-1}
\frac{d^+\eta}{dt} \leq C \eta \|\rho\| \left( 1 +
\eta^{d/p^{\prime}-1} \eta^{-\alpha}_m \right) \leq C \eta
\|\rho\| \left( 1 + \eta^{d/p^{\prime}} \eta^{-(1 + \alpha)}_m
\right)\,,
\end{equation}
for all $t\in[0,\min(T,T_0^N))$.

\

\textbf{Step B.-} We now focus on showing the lower bound estimate
of $\eta_m$ to make the system \eqref{ma-re-1} closed. We again
choose two indices $i,j$ so that $|X_i - X_j| = \eta_m$.
Neglecting the time dependency to simplify the notation, we get
\begin{align*}
\begin{aligned}
\frac{d}{dt} | X_i - X_j| &\geq -|u_N(X_i) - u_N(X_j)| \cr &\geq -
\int_{\bbr^d} |\nabla W (X_i - y) - \nabla W (X_j - y)|
\,d\mu_N(y)
 \cr &= - \int_{\bbr^d} |\nabla W (X_i - \mathcal{T} (y)) -
\nabla W (X_j - \mathcal{T} (y))| \,\rho(y) dy\,,
\end{aligned}
\end{align*}
where $\mathcal{T}$ is the optimal map satisfying $\mu_N(t) =
\mathcal{T} \# \rho(t)$, for each $t\in[0,\min(T,T_0^N))$. Similar
to \eqref{mt-est-3}, we split in near- and far-field parts the
domain $\bbr^d$ as $\mathcal{A} := \{ y : |X_i - y| \geq 2\eta
\mbox{ and } |X_j - y| \geq 2\eta \}$ and $\mathcal{B} := \bbr^d -
\mathcal{A}$. We can again use \eqref{key} to deduce
\begin{align}\label{ma-re-1-est-1}
 \int_{\mathcal{A}} |&\nabla W (X_i - \mathcal{T}(y)) - \nabla W (X_j -
\mathcal{T}(y))|\, \rho(y) dy  \\ & \leq \int_{\mathcal{A}}
\frac{2|X_i-X_j|}{\min (|X_i - \mathcal{T} (y)|,|X_j - \mathcal{T}
(y)|)^{\alpha + 1}} \,\rho(y) dy \nonumber \\ & \leq
2^{2+\alpha}|X_i-X_j| \int_{\mathcal{A}} \left(\frac{1}{|X_i - y|^{\alpha +
1}} + \frac{1}{|X_j - y|^{\alpha + 1}}\right) \rho(y) dy \leq C
\eta_m \|\rho\|,\nonumber
\end{align}
where we used that $|X_i - \mathcal{T}(y)| \geq | X_i - y| - \eta
\geq \frac{1}{2} |X_i - y|$ and similarly for $X_j$ together with
\eqref{nearfar}. For the integral over $\mathcal{B}$, we use that as soon as
$X_i \neq \mathcal{T}(y)$, then we obtain from
\eqref{main-assum-1} that
\begin{equation*}
|\nabla W(X_i - \mathcal{T}(y))| \leq \frac{1}{|X_j -
\mathcal{T}(y)|^{\alpha}} \leq \frac{1}{\eta^{\alpha}_m},
\end{equation*}
and $\nabla W(X_i - \mathcal{T}(y))=0$ otherwise, and similarly
for $X_j$. A simple H\"older computation as in \eqref{nearfar}
implies that
\[
\int_{\mathcal{B}} \rho(y) dy \leq C\eta^{d/p^{\prime}} \|\rho\|\,,
\]
from which we infer that
\begin{equation}\label{ma-re-1-est-2}
\int_{\mathcal{B}} |\nabla W (X_i - \mathcal{T} (y)) - \nabla W (X_j -
\mathcal{T} (y))|\, \rho(y) dy \leq C\eta^{d/p^{\prime}}
\eta^{-\alpha}_m\|\rho\| .
\end{equation}
Putting together \eqref{ma-re-1-est-1} and \eqref{ma-re-1-est-2},
we finally conclude that
\begin{equation}\label{ma-re-2}
\frac{d\eta_m}{dt} \geq - C \eta_m \|\rho\| \left( 1 +
\eta^{d/p^{\prime}} \eta_m^{-(1 + \alpha)}\right)\,,
\end{equation}
for all $t\in[0,\min(T,T_0^N))$.

\

\textbf{Step C.-} Until now, we have proved from \eqref{ma-re-1}
and \eqref{ma-re-2} that
\begin{equation}\label{step-c-1}
\left\{  \begin{array}{ll} \displaystyle \frac{d^+\eta}{dt} &\leq
C \eta \|\rho\| \left( 1 + \eta^{d/p^{\prime}} \eta^{-(1 +
\alpha)}_m \right),
\\[3mm]
\displaystyle \frac{d\eta_m}{dt} &\geq - C \eta_m \|\rho\| \left(
1 + \eta^{d/p^{\prime}} \eta_m^{-(1 + \alpha)}\right),
\end{array} \right.
\end{equation}
for $t \in [0,\min(T,T_0^N))$. We first notice from
\eqref{step-c-1} that if $\eta^{d/p^{\prime}} \eta_m^{-(1 +
\alpha)} \leq 1$, then
\begin{equation}\label{ma-re-2-est-1}
\eta(t) \leq \eta^0 e^{2\|\rho\| t} \quad \mbox{and} \quad
\eta_m(t) \geq \eta_m^0 e^{-2\|\rho\| t} \quad t \in
[0,\min(T,T_0^N)).
\end{equation}
We now show that \eqref{ma-re-2-est-1} holds for time $t \in
[0,T]$ when $N$ goes to infinity, in other words that $T < T_0^N$
when $N$ is sufficiently large. For this, we set
\[
f(t) := \frac{\eta(t)}{\eta^0}, \quad g(t) := \frac{\eta_m(t)}{\eta^0_m} \quad \mbox{and} \quad \xi_N := (\eta^0)^{d/p^{\prime}} (\eta_m^0)^{-(1 + \alpha)}.
\]
Note that $\xi_N$ depends on the number of particles $N$ as in
\eqref{notat-simp}. It yields
\begin{align*}
\begin{aligned}
\frac{d^+f}{dt} &\leq C\|\rho\| \,f\,\left(1 + \xi_N f^{d/p^{\prime}}
g^{-(1 + \alpha)}\right),\cr \frac{dg}{dt} &\geq -C\|\rho\|\, g \, \left(1 +
\xi_N f^{d/p^{\prime}} g^{-(1 + \alpha)}\right).
\end{aligned}
\end{align*}
Since $f(0)=g(0)=1$ and $\xi_N \to 0$ as $N$ goes to infinity, we
obtain that there exists a positive constant $T_*^N (\leq T_0^N)$
such that
\[
\xi_N f^{d/p^{\prime}} g^{-(1 + \alpha)} \leq 1 \quad \mbox{for}
\quad t \in [0,T_*^N]\,,
\]
for sufficiently large $N$.
Then it follows from \eqref{ma-re-2-est-1} that
\[
f(t) \leq e^{2\|\rho\| t} \quad \mbox{and} \quad g(t) \geq
e^{-2\|\rho\| t}.
\]
This yields $\xi_N f^{d/p^{\prime}} g^{-(1 + \alpha)} \leq \xi_N
e^{2(d/p^{\prime} + (1 + \alpha))\|\rho\| t}$, that is,
\[
\xi_N f^{d/p^{\prime}} g^{-(1 + \alpha)} \leq 1 \quad \mbox{holds
for} \quad t \leq  -\frac{\ln(\xi_N)}{2(d/p^{\prime} + (1 +
\alpha))\|\rho\|},
\]
so that
\[
-\frac{\ln(\xi_N)}{2(d/p^{\prime} + (1 + \alpha))\|\rho\|} \leq
T_*^N\,.
\]
On the other hand, our assumption for the initial data \eqref{ini-assum} implies
\[
\liminf_{N \to \infty} T_*^N \geq \lim_{N \to \infty}
-\frac{\ln(\xi_N)}{2(d/p^{\prime} + (1 + \alpha))\|\rho\|}
=\infty\,,
\]
and thus for $N$ large enough, $T<T_*^N<T_0^N$. This completes the
proof.
\end{proof}

\begin{remark}
One can use almost the same argument with the above to obtain an
stability estimate in $d_\infty$: let $\rho_1$ and $\rho_2$ be
solutions given by Theorem \ref{BLRlike-thm} to the system
\eqref{conti-agg} satisfying \eqref{main-assum-1}, then we have
\[
\frac{d}{dt} d_{\infty}(\rho_1(t),\rho_2(t)) \leq C
\max(\|\rho_1\|, \|\rho_2\|)\,
d_{\infty}(\rho_1(t),\rho_2(t))\,.
\]
\end{remark}

\vspace{0.2cm} In fact, the estimate of mean field limit in
Theorem \ref{main-thm} holds for $-1\leq\alpha< 0$ without any
condition on $\eta^0$ and $\eta_m^0$. This is coherent with the
results in \cite{CDFLS} in which the mean field limit is obtained
for all measure initial data without restriction in the way
initial data are approximated by Dirac masses at least for
attractive potentials.

\begin{corollary}\label{corol-1}
Suppose the interaction potential $W$ satisfies
\eqref{main-assum-1} with $-1 \leq \alpha < 0$, and let $\rho$ be
a solution to the system \eqref{conti-agg} such that $\rho \in
L^{\infty}(0,T;(L^1 \cap L^p)(\R^d))\cap
\mathcal{C}([0,T],\mathcal{P}_1(\R^d))$. Suppose that
\[
d_{\infty}(\mu_N^0,\rho^0) \to 0 \quad \mbox{as} \quad N \to \infty.
\]
Then for any solution of the ODE system \eqref{disconti-agg} the associated empirical distributions $\mu_N(t)$ converge toward $\rho(t)$ uniformly in time:
\[
\sup_{t \in [0,T]} d_{\infty}(\mu_N(t),\rho(t)) \to 0 \quad \mbox{as} \quad N \to \infty.
\]
\end{corollary}

\begin{remark}
It is remarkable that even if we do not have uniqueness of
solution of \eqref{disconti-agg} under assumption
\eqref{main-assum-1} with $-1 \leq \alpha < 0$, we get the mean
field limit without restriction. If one collision occurs, then
uniqueness may lost, but the existence of solution is still
guaranteed. Thus Corollary \ref{corol-1} is interesting because it
is valid for density solutions to \eqref{conti-agg} even if
collisions occur and uniqueness is lost at the particle level.
\end{remark}

\begin{proof}[Proof of Corollary \ref{corol-1}]
We first notice that the existence of solutions to the ODE system
\eqref{disconti-agg} is guaranteed thanks to Cauchy-Peano-Arzela theorem since
$\alpha$ is strictly negative with \eqref{main-assum-1} implies that $\nabla W$
is continuous. One can
use the
same arguments as in the proof of Theorem \ref{main-thm} to find
\begin{align*}
\begin{aligned}
\frac{d^+\eta}{dt} &\leq C \sup_{x \in \bbr^d} \left(\int_{\mathcal{A}} +
\int_{\mathcal{B}}\right) | \nabla W(\mathcal{T} (x) - \mathcal{T} (y)) -
\nabla W (x-y) |\, \rho(y) dy\cr & := \mathcal{K}_1 +
\mathcal{K}_2,
\end{aligned}
\end{align*}
where the same notation for the sets $\mathcal{A}$ and $\mathcal{B}$ is used and the
time dependency has been avoided for simplicity. Using
\eqref{mt-est-4} we estimate $\mathcal{K}_1$ by $C\eta \|\rho\|$.
To estimate $\mathcal{K}_2$, we use that $\alpha<0$ to get
\[
|\nabla W (\mathcal{T} (x) - \mathcal{T} (y)) - \nabla W (x-y)|
\leq \frac{C}{\eta^{\alpha}} + \frac{C}{|x-y|^{\alpha}}\,,
\]
and to obtain by H\"older's inequality that
\begin{align*}
\begin{aligned}
\mathcal{K}_2 &\leq C\int_{\mathcal{B}} \frac{\rho(y)}{|x-y|^{\alpha}} dy +
\frac{C}{\eta^{\alpha}}\int_{\mathcal{B}} \rho(y) dy \leq
C\eta^{d/p^{\prime}-\alpha} \|\rho\|_{p} + C\eta^{d/p^{\prime}}
\eta^{-\alpha} \|\rho\|_{p} \cr &\leq C\eta^{d/p^{\prime}-\alpha}
\|\rho\| \,.
\end{aligned}
\end{align*}
Hence, we have
\[
\frac{d^+\eta}{dt} \leq C \eta \|\rho\| \left( 1 +
\eta^{d/p^{\prime}-\alpha-1} \right),
\]
and this yields for sufficiently large $N$
\begin{equation*}
\eta(t) \leq \left( (\eta^0)^{1 - (d/p' - \alpha)} e^{-C \|\rho\|
(d/p' - \alpha - 1)t} + e^{-C \|\rho\| (d/p' - \alpha - 1)t} - 1
\right)^{-\frac{1}{d/p' - \alpha - 1}},
\end{equation*}
for all $t \in [0,T]$. Note that $d/p' - \alpha - 1 > 0$ and then,
the right hand side of previous estimate goes to zero as $N$ goes
to infinity. This completes the proof.
\end{proof}

We next show that there is no collision between particles when the
initial quantities $\eta^0$ and $\eta_m^0$ in \eqref{notat-simp}
satisfy
\begin{equation}\label{ini-assum-cor}
\lim_{N \to \infty}
\frac{(\eta^0)^{d/p^{\prime}-\alpha}}{\eta^0_m} = 0.
\end{equation}
Note that the same strategy as in Remark \ref{rem:initialdata}
allows us to find suitable approximations for the initial data
satisfying \eqref{ini-assum-cor}.

\vspace{0.2cm}

\begin{corollary}\label{coroll-2}
Under the assumptions of Corollary \ref{corol-1} with $-1 \le
\alpha < 0$, if we further assume that $\eta^0,~\eta_m^0$ satisfy
\eqref{ini-assum-cor}. Then we have that for $N$ large enough, the
particle system \eqref{disconti-agg} is uniquely well-defined till
time $T$ in the sense that there is no collision between particles
before that time, and the convergence
\[
\sup_{t \in [0,T]} d_{\infty}(\mu_N(t),\rho(t)) \to 0 \quad
\mbox{as} \quad N \to \infty \,,
\]
holds.
\end{corollary}

\begin{proof} The proof of Corollary \ref{corol-1} shows that for sufficiently large $N$
\[
\eta \leq \left( (\eta^0)^{1 - (d/p' - \alpha)} e^{-C \|\rho\| (d/p'
- \alpha - 1)t} + e^{-C \|\rho\| (d/p' - \alpha - 1)t} - 1
\right)^{-\frac{1}{d/p' - \alpha - 1}}.
\]
For the estimate of $\eta_m$, one can obtain from the proof of
Theorem \ref{main-thm} that
\[
\frac{d\eta_m}{dt} \geq - C \eta_m \|\rho\| \left( 1 +
\eta^{d/p^{\prime}-\alpha} \eta_m^{-1}\right) \quad \mbox{for all}
\quad t \in [0,\min(T,T_0^N)),
\]
where $T_0^N$ denotes the first collision time between particles.
Then we conclude the desired result employing the same arguments
in Step C of Theorem \ref{main-thm} using \eqref{ini-assum-cor}.
\end{proof}

As a corollary of Theorem \ref{main-thm}, we consider interaction
potentials under weaker assumptions than \eqref{main-assum-1}:
there exists $R > 0$ such that $W$ satisfies
\begin{equation}\label{local-assum}
|\nabla W(x)|\leq \frac{C}{|x|^{\alpha}}, \quad \mbox{and} \quad
|D^2 W(x)| \leq \frac{C}{|x|^{1+\alpha}}, \quad \forall ~ x \in
B(0,R),
\end{equation}
where $B(0,R) :=\{ x \in \R^d : |x| < R \}$. Then one can assume
that the initial data $\rho^0$ has compact support, and show that
the local solution $\rho(t)$ has compact support on a small time
interval $[0,T]$. This is possible since characteristics are
locally in time well defined and the velocity is uniformly bounded
under the assumptions \eqref{local-assum} initially. This argument
was made rigorous under stricter assumptions on the local
behaviour of the interaction potential but allowing growth of the
potential at infinity in \cite{B-C}. Thus, one can cut-off the
potential outside a large ball in such a way that the solution is
unaffected but the potential satisfies the global assumption
$\nabla W \in \mathcal{W}^{1,p^{\prime}}(\R^d)$ entering the
well-posedness theory in \cite{B-L-R} or satisfying
\eqref{main-assum-1} allowing for the application of Theorem
\ref{BLRlike-thm}. Concerning the interaction potential $W$
satisfying \eqref{local-assum}, the same results of convergence in
Theorem \ref{main-thm} and Corollary \ref{coroll-2} can be
obtained. We leave the details to the reader.


\section{Local existence and uniqueness of
$L^p$-solutions}\label{sec-exist} In this section, we provide a
local existence and uniqueness result of weak solutions in
$L^p$-spaces to the system \eqref{conti-agg} under the assumptions
\eqref{main-assum-1}.

As we mentioned before, we can not directly apply the arguments in
\cite{B-L-R} for the potentials satisfying \eqref{main-assum-1}.
Of course, we can overcome these difficulties using the property
of compact supports on the initial data $\rho^0$ (see the
paragraph below Corollary \ref{coroll-2}). However, we use the
arguments of dividing near- and far-field parts of the interaction
potential function $W$ to establish the local existence of a
unique $L^p$-solution to the continuity aggregation equation
\eqref{conti-agg}.

\begin{theorem}\label{BLRlike-thm}
Assume that $W$ satisfies the condition~\eqref{main-assum-1}, for
some $0\leq\alpha < \frac d {p'} -1$, and that $\rho^0 \in
\mathcal{P}_1(\R^d) \cap L^p(\R^d)$, $1<p\leq \infty$. Then there
exists a time $T > 0$, depending only on $\|\rho^0\|_{p}$ and
$\alpha$, and a unique nonnegative solution to~\eqref{conti-agg}
satisfying $\rho \in L^\infty(0,T;L^1 \cap L^p(\bbr^d))\cap
\mathcal{C}([0,T],\mathcal{P}_1(\R^d))$. Furthermore, the solution
satisfies that there exists $C>0$ depending only on
$\|\rho^0\|_{p}$ and $\alpha$ such that
\begin{equation}\label{boundlp}
\|\rho(t)\|_p\leq C\,\qquad \mbox{for all } t\in[0,T].
\end{equation}
The velocity field generated by $\rho$, given by $u=-\nabla W\ast
\rho$, is bounded and Lipschitz continuous in space uniformly on $[0,T]$, and
$\rho$ is determined as the push-forward of the initial density through
the flow map generated by $u$.

Moreover, if  $\rho_i$, $i=1,2$, are two such solutions
to~\eqref{conti-agg} with initial conditions $\rho_i^0 \in
\mathcal{P}_1(\bbr^d) \cap L^p(\bbr^d)$, $1<p \leq \infty$, we
have the following stability estimate:
\[
\frac{d}{dt} d_1(t) \leq C \max( \|\rho_1\|,\|\rho_2\| ) d_1(t),
\]
where $d_1(t) := d_1(\rho_1(t),\rho_2(t))$.
\end{theorem}
\begin{proof}
Let us start by proving the uniqueness. Given two weak
solutions
$\rho_i \in L^\infty(0,T;L^1 \cap L^p(\bbr^d))\cap
\mathcal{C}([0,T],\mathcal{P}_1(\R^d))$, $i=1,2$, to the continuous
aggregation equations \eqref{conti-agg}, consider the two flow
maps $\Psi_i : \bbr_+ \times \bbr_+ \times \bbr^d \to \bbr^d$,
$i=1,2$, generated by the two velocity fields, i.e.,
\begin{equation*}
\left\{  \begin{array}{ll}
\displaystyle \frac{d}{dt} \left(\Psi_i(t;s,x)) = u_i(t;s,\Psi_i(t;s,x)\right), & \\[2mm]
\Psi_i(s;s,x) = x, &
\end{array} \right.
\end{equation*}
where $u_i := -\nabla W * \rho_i$, $t,s\in [0,T]$ and $x\in\R^d$.
We know that the solutions are constructed by
transporting the
initial measures through the velocity fields $\rho_i = \Psi_i \#
\rho_i^0$, $i=1,2$.

Let $\mathcal{T}^0$ be the optimal transportation between
$\rho_1(0)$ and $\rho_2(0)$ for the $d_1$-distance. Then we define
a transport (not necessarly optimal) between $\rho_1(t)$ and
$\rho_2(t)$ by
\[
 \mathcal{T}^t(x) =
\Psi_2(t;0,x) \circ \mathcal{T}^0(x) \circ \Psi_1(0;t,x), \qquad
\mathcal{T}^t \# \rho_1(t) = \rho_2(t),
\]
and $\frac{d}{dt}d_1(t) \leq Q(t)$, where $d_1(t) :=
d_1(\rho_1(t),\rho_2(t))$ and
\[
Q(t) := \int_{\bbr^d \times \bbr^d} |\nabla W(\mathcal{T}^t (x) -
\mathcal{T}^t (y)) - \nabla W(x-y)| \rho_1(t,x)\rho_1(t,y) dxdy\,,
\]
where we have used a similar argument as in Step A of the proof of
Theorem \ref{main-thm}. To simplify the notation, let us not make
explicit the dependence on time. Note by symmetry that
\begin{align*}
\begin{aligned}
Q(t) &\leq 4\int_{\bbr^d \times \bbr^d} \left(
\frac{|\mathcal{T}(x) - x|}{|\mathcal{T}(x) - \mathcal{T}
(y)|^{1+\alpha}} + \frac{|\mathcal{T}(x) -
x|}{|x-y|^{1+\alpha}}\right)\rho_1(x)\rho_1(y) dxdy\cr &:=
\mathcal{J}_1 + \mathcal{J}_2.
\end{aligned}
\end{align*}
Straightforward computation using the near- and far-field
decomposition as in \eqref{nearfar} shows that
\begin{align*}
\begin{aligned}
\mathcal{J}_1 &= 4\int_{\bbr^d} |\mathcal{T} (x) - x |\rho_1(x)
\left( \int_{\bbr^d} \frac{\rho_2(y)}{|\mathcal{T} (x) -
y|^{1+\alpha}} dy \right) dx \cr &\leq C \|\rho_2\| \int_{\bbr^d}
|\mathcal{T} (x) - x |\rho_1(x) dx =C \|\rho_2\| \,d_1(t).
\end{aligned}
\end{align*}
Similarly using again \eqref{nearfar}, we have $\mathcal{J}_2 \leq
C \|\rho_1\|d_1(t)$. It yields that
\[
\frac{d}{dt}d_1(t) \leq C\max(\|\rho_1\|, \|\rho_2\|)\,d_1(t)\,,
\]
from which we conclude the uniqueness part of the statement.

\medskip

Let us now show the existence of weak solution. Let $\ep > 0$ and
$\theta$ be a standard mollifier:
\[
\theta \geq 0, \quad \theta \in \mathcal{C}_0^{\infty}(\R^d), \quad \mbox{supp }\theta \subset B(0,1), \quad \int_{\R^d} \theta(x) dx = 1,
\]
and we set a sequence of smooth mollifiers:
\[
\theta_{\ep}(x) := \frac{1}{\ep^d}\theta \left( \frac{x}{\ep}\right).
\]
We first regularize $\nabla W$ such as $\nabla W_{\ep} := (\nabla
W) * \theta_{\ep}$. Then since $\nabla W_{\ep}$ is a globally
Lipschitz, we can apply the theory of \cite{BH,dobru,L} which says
that there exists a unique global solution $\rho_{\ep}$ to the
following system
\begin{equation}\label{regularprob}
\left\{ \begin{array}{ll}
\partial_t \rho_{\ep} + \nabla \cdot (\rho_{\ep} u_{\ep}) = 0, & \qquad t > 0, \quad x \in
\bbr^d,\\[2mm]
u_{\ep}(t,x) := - \nabla W_{\ep} * \rho_{\ep}, & \qquad t > 0, \quad x \in
\bbr^d,\\[2mm]
\rho_{\ep}(0,x) := \rho^0(x), & \qquad x \in \bbr^d,
\end{array} \right.
\end{equation}
A standard calculation, see \cite{B-L-R}, implies that
\begin{equation}\label{de-unif-bd}
\frac{d}{dt}\|\rho_{\ep}\|_{L^1 \cap L^p} \leq C \|\rho_{\ep}\|^2_{L^1 \cap
L^p},
\end{equation}
where $C$ is an uniform constant in $\ep$. Note that the
inequality \eqref{de-unif-bd} holds only formally for the non
regularized problem, but it is fully rigorous for the regularized
one with $W_{\ep}$. This yields that the time of blow-up depends
only on the initial data, more precisely $\| \rho^0\|$, and not on
$\ep$. Thus, there exists a $T>0$ such that
\begin{equation}\label{unif-bdd}
\sup_{\ep > 0} \|\rho_{\ep}\| < \infty.
\end{equation}
It follows from \eqref{unif-bdd} and the evolution in time of the
first momentum of $\rho$, that this first moment is also uniformly
bounded:
\[
\sup_{\ep > 0}\| x \rho_{\ep}\|_{L^{\infty}(0,T;L^1(\R^d))} \leq C,
\]
where $C$ depends only on $T, \|x \rho^0\|_1$, and $\|\rho^0\|$.
We leave the details to the reader. Next, we show an estimate on
the growth of the $d_1$ distance $\eta_{\ep,\ep'}(t) :=
d_1(\rho_{\ep}(t),\rho_{\ep'}(t))$ between $\rho_{\ep}$ and
$\rho_{\ep'}$, for $\ep,\ep' > 0$:
\begin{equation}\label{reg-d-1-est}
\frac{d}{dt}\eta_{\ep,\ep'}(t) \leq C\max( \|\rho_{\ep}\|,
\|\rho_{\ep'}\|)\left(\eta_{\ep,\ep'}(t) + \ep + \ep'
\right),
\end{equation}
where $C$ is an uniform constant in $\ep$ and $\ep'$. We remark
that the above estimate \eqref{reg-d-1-est} implies that
$\{\rho_{\ep}\}_{\ep>0}$ is a Cauchy sequence in
$\mathcal{C}([0,T],\mathcal{P}_1(\R^d))$.

Let us remark that the weak solutions to the regularized problems
\eqref{regularprob} can be written in terms of characteristics.
This is a consequence of the fact that the associated velocity
field $u_\epsilon$ is bounded and Lipschitz in space, unifromly in time  and
some standard duality arguments. This strategy is explained in detail
at the end of the proof of the present Theorem applied to the
solution of the original problem, and we refer the reader there
for details. Since solutions are constructed by characteristics,
for the proof of \eqref{reg-d-1-est} we can proceed as in the part
of uniqueness. Therefore, not making explicit the time dependency,
we get
\begin{align}
\frac{d}{dt}\eta_{\ep,\ep'}(t) &\leq \int_{\R^d \times \R^d}
|\nabla W_{\ep}(\mathcal{T} (x) - \mathcal{T} (y)) - \nabla
W_{\ep'}(x-y)|\,\rho_{\ep'}(x)\rho_{\ep'}(y) dx dy\nonumber\\
&\leq \int_{\R^d \times \R^d} |\nabla W_{\ep}(\mathcal{T} (x) -
\mathcal{T} (y)) - \nabla
W_{\ep}(x-y)|\, \rho_{\ep'}(x)\rho_{\ep'}(y) dx dy \nonumber\\
&\quad + \int_{\R^d \times \R^d} |\nabla W_{\ep}(x-y) - \nabla
W_{\ep'}(x-y)|\,\rho_{\ep'}(x)\rho_{\ep'}(y) dx dy\nonumber\\
&:=\mathcal{K}_1 + \mathcal{K}_2, \label{reg-d-1-est-1}
\end{align}
where $\mathcal{T}$ is the optimal transportation between
$\rho_{\ep'}(t)$ and $\rho_{\ep}(t)$ for the $d_1$-distance. To
estimate $\mathcal{K}_1$, we notice that
\begin{align}
|\nabla W_\ep(x)| &\leq \int_{\left\{y: |y| <
\frac{|x|}{2}\right\}} \frac{\theta_{\ep}(y)}{|x-y|^{1+\alpha}} dy
+ \int_{\left\{y: |y| \geq \frac{|x|}{2}\right\}} \frac{
\theta_{\ep}(y)}{|x-y|^{1+\alpha}}  dy \nonumber
\\ &\leq
\frac{2^{1+\alpha}}{|x|^{1+\alpha}}\int_{\R^d}\theta_{\ep}(y)dy +
\mathbf{1}_{\{|x| \leq 2\ep\}}\int_{\left\{y:~ \ep \geq |y|
\right\}}\frac{\theta_{\ep}(y)}{|x-y|^{1+\alpha}}dy \nonumber
\\&\leq \frac{C}{|x|^{1+\alpha}} +
\frac{C\ep^{1+\alpha}}{|x|^{1+\alpha}}\int_{\left\{ y:~ \ep \geq
|y| \right\}}\frac{\theta_{\ep}(y)}{|x-y|^{1+\alpha}} dy \leq
\frac{C}{|x|^{1+\alpha}}. \label{reg-d-1-est-111}
\end{align}
Then we now use again the decomposition
\eqref{nearfar} as in the part of uniqueness to find
\begin{equation}\label{reg-d-1-est-2}
\mathcal{K}_1 \leq C\max(\|\rho_{\ep}\|, \|\rho_{\ep'}\| )\,
\eta_{\ep,\ep'}(t)\,,
\end{equation}
where $C$, $\|\rho_{\ep}\|$, and $\|\rho_{\ep'}\|$ are uniformly
bounded in $\ep$ and $\ep'$ thanks to the estimate
\eqref{unif-bdd}. For the estimate of $\mathcal{K}_2$, we claim
that
\begin{equation}\label{useful-est}
|\nabla (W - W_{\ep})(x)| \leq \frac{C\ep}{|x|^{1+\alpha}},
\end{equation}
where $C$ is independent on $\ep$.

\textit{Proof of Claim}: It is a straightforward to obtain
\begin{align}\label{useful-est-1}
\begin{aligned}
|\nabla W_{\ep}(x) - \nabla W (x)| &\leq \int_{\R^d} |\nabla
W(x-y) - \nabla W(x)|\theta_{\ep}(y) dy \cr &\leq 2 \int_{\R^d}
\left( \frac{1}{|x|^{1+\alpha}} + \frac{1}{|x-y|^{1+\alpha}}
\right) |y| \theta_{\ep}(y) dy \cr & := \mathcal{L}_1 +
\mathcal{L}_2\,.
\end{aligned}
\end{align}
Noticing that the mollifier properties allow to gain an $\ep$
factor in front of the integrals, we can estimate
$\mathcal{L}_i,i=1,2$ as follows
\begin{align}\label{useful-est-2}
\begin{aligned}
\mathcal{L}_1 &\leq \frac{C\ep}{|x|^{1+\alpha}}\int_{\R^d}
\theta_{\ep}(y) dy = \frac{C\ep}{|x|^{1+\alpha}}, \cr
\mathcal{L}_2 &\leq 2\ep\int_{\R^d}\frac{\theta_{\ep}(y)}{|x-y|^{1+\alpha}}dy \leq
\frac{C\ep}{|x|^{1+\alpha}},
\end{aligned}
\end{align}
where we used a similar argument to \eqref{reg-d-1-est-111} for
$\mathcal{L}_2$. We now combine \eqref{useful-est-1} and
\eqref{useful-est-2} to have the inequality \eqref{useful-est}.
Then we use \eqref{useful-est} together with \eqref{nearfar} to
find the estimate of $\mathcal{K}_2$
\begin{equation}\label{reg-d-1-est-3}
\mathcal{K}_2 \leq C(\ep + \ep')\int_{\R^d \times \R^d}
\frac{\rho_{\ep'}(t,x)\rho_{\ep'}(t,y)}{|x-y|^{1+\alpha}}  dx dy
\leq C(\ep + \ep')\|\rho_{\ep'}\|\,.
\end{equation}
This completes the proof of the inequality \eqref{reg-d-1-est} by
combining \eqref{reg-d-1-est}, \eqref{reg-d-1-est-1},
\eqref{reg-d-1-est-2}, and \eqref{reg-d-1-est-3}.

Since $\rho_{\ep}$ is a Cauchy sequence in
$\mathcal{C}([0,T],\mathcal{P}_1(\R^d))$, it converges toward a
limit curve of measures $\rho \in
\mathcal{C}([0,T],\mathcal{P}_1(\R^d))$, and we also have $\rho
\in L^{\infty}(0,T;L^1 \cap L^p(\R^d))$ from the uniform bounded
estimate \eqref{unif-bdd}. It remains to show that $\rho$ is a
solution of the aggregation equations \eqref{conti-agg}. Choose a
test function $\phi(t,x) \in \mathcal{C}_c^{\infty}([0,T]\times
\R^d)$, then $\rho_{\ep}$ satisfies
\begin{align}\label{reg-limit}
\int_{\R^d} \!\!\!\rho^0(x) \phi^0(x) dx =& \int_{\R^d}
\!\!\!\rho_\ep(T,x) \phi(T,x) dx +
\!\!\int_0^T\!\!\!\!\int_{\R^d}\!\! \rho_{\ep}(t,x)\partial_t
\phi(t,x) dx dt \\ & - \!\!\int_0^T\!\!\!\int_{\R^d}\!\int_{\R^d}
\rho_{\ep}(t,x)\rho_{\ep}(t,y)\nabla W_{\ep}(x-y)\cdot \nabla
\phi(t,x) dx dy dt.\nonumber
\end{align}
The first two terms in the rhs of \eqref{reg-limit} converges to
\[
\int_{\R^d} \rho(T,x) \phi(T,x) dx + \int_0^T\int_{\R^d}
\rho(t,x)\partial_t \phi(t,x) dx dt,
\]
since $\rho_{\ep} \to \rho$ in $\mathcal{C}([0,T],\mathcal{P}_1(\R^d))$. For
the third term in the rhs of \eqref{reg-limit}, we use the
estimates \eqref{useful-est} and \eqref{unif-bdd} to find
\begin{align*}
\left|\int_0^T\!\!\!\!\int_{\R^d}\!\int_{\R^d}\!\!
\rho_{\ep}(t,x)\rho_{\ep}(t,y)\left(\nabla W_{\ep}(x-y)-\nabla
W(x-y)\right)\cdot \nabla \phi(t,x) dx dy dt \right| \to 0,
\end{align*}
as $\ep \to 0$. It remains to show that
\begin{align*}
\begin{aligned}
&\int_0^T\int_{\R^d}\int_{\R^d}
\rho_{\ep}(t,x)\rho_{\ep}(t,y)\nabla W(x-y)\cdot \nabla \phi(t,x)
dx dy dt \cr & \qquad\qquad \qquad \to
\int_0^T\int_{\R^d}\int_{\R^d} \rho(t,x)\rho(t,y)\nabla
W(x-y)\cdot \nabla \phi(t,x) dx dy dt,
\end{aligned}
\end{align*}
as $\ep \to 0$. For this, we introduce a cut-off function
$\chi_{\delta} \in \mathcal{C}_c^{\infty}(\R)$ such that
\[
\chi_{\delta}(x) = \left\{  \begin{array}{ll}
1 & \mbox{if} \quad |x|\leq \delta \\
0 & \mbox{if} \quad |x|\geq 2\delta
\end{array} \right.\,.
\]
Then it follows from the weak convergence that \begin{align*}
\begin{aligned}
&\int_0^T\!\!\int_{\R^d}\!\int_{\R^d}\!\!
\rho_{\ep}(t,x)\rho_{\ep}(t,y)(1 - \chi_{\delta}(x-y))\nabla
W(x-y)\cdot \nabla \phi(t,x) dx dy dt \cr & \qquad \to
\int_0^T\!\!\int_{\R^d}\!\int_{\R^d}\!\! \rho(t,x)\rho(t,y)(1 -
\chi_{\delta}(x-y))\nabla W(x-y)\cdot \nabla \phi(t,x) dx dy dt,
\end{aligned}
\end{align*}
as $\ep \to 0$, since $(1 - \chi_{\delta}(x-y))\nabla W(x-y)\cdot
\nabla \phi(t,x)$ is a Lipschitz function. We estimate the
remainder as follows:
\begin{align*}
\begin{aligned}
&\left|\int_0^T\!\!\int_{\R^d}\!\int_{\R^d}\!\!
\rho_{\ep}(t,x)\rho_{\ep}(t,y)\chi_{\delta}(x-y))\nabla
W(x-y)\cdot \nabla \phi(t,x) dx dy dt\right| \cr & \qquad \quad
\leq C\delta\int_0^T\int_{\{(x,y)\in \R^d \times \R^d :~|x-y|\leq
2\delta\}} \frac{1}{|x-y|^{1+\alpha}}
\rho_{\ep}(t,x)\rho_{\ep}(t,y) dx dy dt \cr & \qquad \quad \leq
CT\delta\|\rho_{\ep}\| \to 0 \quad \mbox{as} \quad \delta \to 0.
\end{aligned}
\end{align*}
Similarly, we have
\[
\lim_{\delta\to 0}\left|\int_0^T\int_{\R^d}\int_{\R^d}
\rho(t,x)\rho(t,y)\chi_{\delta}(x-y)\nabla W(x-y)\cdot \nabla
\phi(t,x) dx dy dt\right| = 0.
\]
Hence, we conclude that $\rho$ satisfies
\begin{align}\label{ws}
\int_{\R^d} \!\!\!\rho^0(x) \phi^0(x) dx &= \int_{\R^d}\!\!\!
\rho(T,x) \phi(T,x) dx + \!\int_0^T\!\!\!\int_{\R^d}
\!\!\!\rho(t,x)\partial_t \phi(t,x) dx dt \\ &\quad -
\int_0^T\!\!\!\int_{\R^d}\int_{\R^d}\!\! \rho(t,x)\rho(t,y)\nabla
W(x-y)\cdot \nabla \phi(t,x) dx dy dt,\nonumber
\end{align}
for all $\phi \in \mathcal{C}_c^{\infty}([0,T]\times \R^d)$.

\medskip
Now, We notice that a weak solution in $\rho \in L^\infty(0,T;L^1
\cap L^p(\bbr^d))$ to \eqref{conti-agg} under the assumptions
\eqref{main-assum-1} has a well defined flow by using the same
arguments as the ones at the beginning of Theorem \ref{main-thm}.
In fact, the velocity field is bounded and Lipschitz continuous in space
 with
$$
|u(t,x)-u(t,y)|\leq C \|\rho\| |x-y|
$$
for all $x,y\in\bbr^d$ and $t \in[0,T]$. Thus, the flow map
\begin{equation*}
\left\{ \begin{array}{ll} \displaystyle \frac{d}{dt} (\Psi(t;s,x))
= u(t;s,\Psi(t;s,x)), &
\\[2mm]
\Psi(s;s,x) = x, &
\end{array} \right.
\end{equation*}
for all $s,t \in [0,T]$ is well-defined. Choosing as test function
in \eqref{ws} $\phi(t,x)=\varphi(\Psi(t;\bar{T},x))$ for any $\bar{T} \in (0,T]$ with $\varphi \in
\mathcal{C}_c^{\infty}(\R^d)$, it is a straightforward to check, due
to the definition of the flow map, that
$$
\int_{\R^d} \rho^0(x) \varphi(\Psi(0;\bar{T},x)) dx = \int_{\R^d}
\rho(\bar{T},x)\varphi(x) dx,
$$
for all $\varphi \in \mathcal{C}_c^{\infty}(\R^d)$, and thus by a
density argument we conclude $\rho(\bar{T}) = \Psi(\bar{T};0,\cdot ~ ) \#
\rho^0$. Since this argument can be done for all $0<\bar{T}\leq
T$, this completes the proof.
\end{proof}

%
%
%
%
\section{Propagation of chaos}\label{sec-propa}

In most practical purposes to approximate the continuum model by
particle systems, it is naturally expected that initial positions
and velocities will randomly and independently be selected. We
will show that the empirical measure at time $0$ is then close to
$\rho^0$ with large probability in suitable weak norm.

In a seminal article \cite{Kac}, the propagation of chaos was
introduced by Kac giving a proof for a simplified collision
evolution process. He showed how the limit of many particles
rigorously follows from the property of propagation of chaos. For
a classical introduction to these topics, we refer to \cite{Szn}.
Later, this property has been studied and developed in kinetic
theory, \cite{McK,McK2,G-M,Hau-Mis,M-M}.

Let us introduce the notion of propagation of chaos. Let us
consider $\rho^N(t,x_1,\cdots, x_N)$ being the image by the
dynamics to the coupled system \eqref{disconti-agg} with $N$-equal
masses particles of the initial law $(\rho^0)^{\otimes N}$. We
define the $k$-marginals as follows.
\[
\rho_k^N(t,x_1,\cdots,x_k) := \int_{\bbr^{d(N-k)}}
\rho_N(t,x)dx_{k+1}\cdots,dx_{N}.
\]
Let us choose the initial positions $X^{N,0}:=\{X^0_i\}_{i=1}^N$
as independent identically distributed random variables (in short
iid) with law $\rho^0$. We can construct the associated empirical
measure as in \eqref{dirac-sol-dis} by
$$
\mu_N(t) = \frac1N \sum_{i=1}^{N} \delta_{X_i(t)}\,,
$$
but now understood as a random variable with values in the space
of probability measures.

The propagation of chaos property is defined as follows: for any
fixed $k\in\N$,
\begin{equation*}
\rho_k^N \rightharpoonup (\rho)^{\otimes k} \quad \mbox{weakly-$*$
as measures as} \quad N \to \infty.
\end{equation*}
It is classically known \cite{Szn} that it is sufficient to check
this property for $k=2$ to derive the propagation of chaos. In
fact, this is based on the fact that propagation of chaos is
equivalent to show that the empirical measures $\mu_N(t)$ converge
in law towards the constant random variable $\rho(t)$.

Theorem \ref{chaos-thm} gives a quantified version of the
convergence in probability of $\mu_N(t)$ towards $\rho(t)$. We
refer to \cite{Hau-Mis,M-M} for a detailed explanation of the
quantified equivalence relations. The propagation of chaos for the
Vlasov-Poisson equations with singular force has recently been
investigated in \cite{Hau-Ja}. Here, we are only able to provide
such a result in a more restrictive setting that in the previous
section. Namely, we only show the propagation of chaos for $d\ge
3$ and with a more restrictive condition on the allowed
singularities $\alpha\geq 0$ depending on the regularity of the
initial data $1<p< \infty$.

\begin{theorem}\label{chaos-thm}
Given $\rho(t) \in L^{\infty}(0,T;(L^1 \cap L^p)(\bbr^d))\cap
\mathcal{C}([0,T],\mathcal{P}_1(\R^d))$ the unique solution to
\eqref{conti-agg} with initial data $\rho^0 \in
\mathcal{P}_1(\R^d) \cap L^p(\R^d)$, $1<p\leq\infty$, up to time
$T>0$. Assume that $\rho^0$ has compact support, that the initial
positions $X^{N,0}:=\{X^0_i\}_{i=1}^N$ are iid with law $\rho^0$,
and that
$$
(1+ \alpha)  p' < \frac{p-1}{2p-1} d\,,
$$
with $\alpha\geq 0$. Then the propagation of chaos holds in the
sense that
$$
\PP \left( \sup_{t \in [0,T]} d_1( \mu_N(t), \rho(t)) \geq
\frac{C}{N^{\gamma/d}} \right) \rightarrow 0 , \quad \text{as } \;
N \rightarrow + \infty,
$$
where $\gamma$ is a positive constant satisfying
\[
\frac{p'(2p-1)(1+\alpha)}{d(p-1)} < \gamma < 1.
\]
\end{theorem}

\begin{remark}
The condition on $\alpha$ gets more and more restrictive as $p$
gets smaller and smaller. In $d=2$, even for $p=\infty$ the
condition is empty for $\alpha \geq 0$. In $d=3$, you get the
condition $\alpha < 1/2$ for $p=\infty$ and with $p=\frac{5 +
\sqrt{13}}2$ the condition is already empty. We also notice that
the existence and uniqueness of the solutions are guaranteed by
Theorem \ref{BLRlike-thm} and taking expectations in the
corresponding inequalities for the particle system. Finally, in
case $-1 \le \alpha <0$, the propagation of chaos holds using the
same strategy as in Corollary \ref{corol-1} by taking expectations
in the inequalities for the evolution of the Wasserstein distance.
\end{remark}

We will follow the strategy recently introduced in \cite{Hau-Ja}
for the Vlasov equation. We first find a deterministic version of
the propagation of chaos. This means that we consider a
regularized system of particles as a kind of middle ground between
the solution of the mean-field equation \eqref{conti-agg} and the
random particle evolution. More precisely, we define the ``blob''
initial data $\rho_{N}^0$ as
\begin{equation} \label{def:rhoNep} \rho_{N}^0 := \mu_N^0 *
\frac{\mathbf{1}_{B_{\varepsilon}(0)}}{|B_{\varepsilon}(0)|}=
\frac{1}{c_d \varepsilon^d}\left(\mu_N^0 *
\mathbf{1}_{B_{\varepsilon}(0)}\right)\,,
\end{equation}
where $\varepsilon>0$ to be chosen as a function of the number of
particles $N$ and $c_d$ is the volume of the unit ball in
dimension $d$. We also define the ``blob'' approximation
$\rho_{N}(t)$ to be the solution of the system \eqref{conti-agg}
with the kernel $W$ satisfying \eqref{main-assum-1} given by
Theorem \ref{BLRlike-thm} and ``blob'' initial data $\rho_{N}^0$.

In the rest, $\varepsilon$ is chosen as a function of $N$ as
$\varepsilon(N) = N^{-\gamma/d}$ with $0 < \gamma < 1$. It is easy
to check that $\|\rho^0_{N}\|_p \simeq N^{(\gamma-1)/p'}$ for $N$
large enough, then we can wonder how far is the empirical measure
to its blob approximation if we assume a bound on
$\|\rho^0_{N}\|_p$ independent of $N$.

\begin{proposition}\label{prop-chao}
Under the assumptions of Theorem \ref{chaos-thm} and assuming that
there exists $C_1>0$ independent of the number of particles $N$
such that
\[
\|\rho^0_{N}\|_p \leq C_1, \quad \text{and} \quad \eta_m^0\geq
\frac 1{C_1} \varepsilon^{r},
\]
with $1 \leq r < \frac{d}{p^{\prime}(1+\alpha)}$. Then, there
exists $T>0$ such that the solutions $\rho_N(t)$ and the empirical
measure $\mu_N(t)$ are well-defined for all $t\in [0,T]$, and
\[
d_{\infty}(\rho_{N}(t),\mu_N(t)) \leq
d_{\infty}(\rho_{N}^0,\mu_N^0)e^{C_2T} \leq
\varepsilon(N)e^{C_2T},
\]
where $C_2>0$ is independent of $N$.
\end{proposition}
\begin{proof}
We follow a similar argument to Theorem \ref{main-thm}. We first
notice from Theorem \ref{BLRlike-thm} that there exists a common
time of existence $T>0$ of the solutions $\rho_N$ independent of
$N$ since it only depends on $\|\rho^0_{N}\|_p$ and $\alpha$. The
empirical measure also exists up to this time since it will be
smaller than the possible first collision time of particles.
Moreover, due to \eqref{boundlp}, we get that $\|\rho_N(t)\|_p
\leq C$, for all $t\in [0,T]$, where $C$ is independent of $N$. We
next substitute $\rho_{N}(t)$ for $\rho(t)$ in the proof of
Theorem \ref{main-thm}, and thus all estimates in Step A and B
hold to deduce
\begin{align*}
\begin{aligned}
    \frac{d\eta_N}{dt} &\leq C \eta_N \|\rho_{N}\| \left( 1 + \eta_N^{d/p^{\prime}} \eta^{-(1 + \alpha)}_m \right) \leq C\eta_N\left( 1 + \eta_N^{d/p^{\prime}} \eta^{-(1 + \alpha)}_m \right),
\end{aligned}
\end{align*}
and
\begin{align*}
\begin{aligned}
\frac{d\eta_m}{dt} &\geq - C \eta_m \|\rho_{N}\| \left( 1 +
\eta_N^{d/p^{\prime}} \eta_m^{-(1 + \alpha)}\right) \geq - C
\eta_m \left( 1 + \eta_N^{d/p^{\prime}} \eta_m^{-(1 +
\alpha)}\right),
\end{aligned}
\end{align*}
where $\eta_N(t) := d_{\infty}(\rho_{N}(t),\mu_N(t))$. Note that
the condition $r \geq 1$ makes sense since $\varepsilon \approx
\eta_N^0 \geq \eta_m^0 \geq C\varepsilon^{r}$ for $\varepsilon$
small enough. We finally conclude the desired result using a
similar argument as in Step C of the proof of Theorem
\ref{main-thm} since
\[
\frac{(\eta_N^0)^{d/p^{\prime}}}{(\eta_m^0)^{1+\alpha}} \leq
C\varepsilon^{d/p^{\prime} - r(1 + \alpha)} \to 0 \quad \mbox{as}
\quad N \to \infty,
\]
by assumption.
\end{proof}

We now present two propositions showing that the assumptions on
$\rho^0_{N}$ and $\eta_m^0$ in Proposition \ref{prop-chao} are
generic in a probability sense when the initial positions
$X^{N,0}$ are iid with law $\rho^0$ in $L^p$. We first prove in
Proposition \ref{prop:mindist} that $\eta_m^0$ is roughly larger
than $N^{- \frac{2p-1}{d(p-1)}}$ if the $X^{N,0}$ are iid with law
$\rho^0$.

\begin{proposition} \label{prop:mindist}
Let $\rho^0 \in \mathcal{P}_1(\R^d) \cap L^p(\R^d)$, $1<p
\leq\infty$, and the initial positions $X^{N,0}$ be iid with law
$\rho^0$. Suppose there exists $L > 0$ such that
\[
2 c_d^\frac{1}{p^{\prime}} \| \rho^0\|_p L^{\frac{d}{p^{\prime}}}
\le N\,,
\]
then $\eta_m^0$ satisfies
$$
\bbp \left( \eta_m^0 \ge L N^{- \frac{2p-1}{d(p-1)}} \right) \geq
e^{- 2c_d^\frac{1}{p^{\prime}} \|\rho^0\|_p
L^{\frac{d}{p^{\prime}}}}\,.
$$
\end{proposition}

\begin{proof}
Choose an $r \in \bbr_+$. Then $\eta_m^0 \ge r$ holds if
$$
X^0_k \in \bbr^d \bigl\backslash A_k\,, \qquad \mbox{with}\qquad
A_k=\bigcup_{1\leq i \le k-1} B(X^0_i,r)\,,
$$
for all $1 \le k \le N$. It implies from our assumption with $r =
L N^{- \frac{2p-1}{d(p-1)}}$ that
\begin{align*}
\bbp \left( \eta_m^0 \ge L N^{- \frac{2p-1}{d(p-1)}}\right) & \ge \prod_{k=1}^N \biggl[ 1 - \int_{A_k}\rho^0(x)\,dx \biggr] \\
& \ge \prod_{k=1}^{N-1} \biggl[ 1 - c_d^\frac{1}{p^{\prime}}
\|\rho^0\|_p L^{\frac{d}{p'}}N^{-2+\frac{1}{p}}
k^{\frac{1}{p^{\prime}}}\biggr] \,,
\end{align*}
and thus using that $\ln(1-x) \ge -2x$ if $x \in [0,\frac12]$, we
conclude
\begin{align*}
\ln \bbp ( \eta_m^0 \ge r) & \ge - 2 c_d^\frac{1}{p^{\prime}} \|
\rho^0\|_p L^{\frac{d}{p'}}N^{-2+\frac{1}{p}} \sum_{k=1}^{N-1}
k^{\frac{1}{p^{\prime}}}  \ge -2c_d^\frac{1}{p^{\prime}} \|
\rho^0\|_p L^{\frac{d}{p^{\prime}}}\,.
\end{align*}
\end{proof}

The next proposition gives some bound on the large deviation of
$\| \rho^0_{N}\|_p$. It states roughly that $\| \rho^0_{N}\|_p$ is
of the same order that $\| \rho^0 \|_p$, if the $X^{N,0}$ are iid
with law $\rho^0$.

\begin{proposition} \label{prop:normLp}
Let $\rho^0 \in \mathcal{P}_1(\R^d) \cap L^p(\R^d)$, $1<p
\leq\infty$, with compactly support included in $[-R,R]^d$. For
any iid $X^{N,0}$ with law $\rho^0$, the smoothed empirical
measures $\rho^0_{N}$ defined in~\eqref{def:rhoNep} satisfy the
explicit ``large deviations'' bound
\begin{equation*} 
\PP \bigl( L_d\| \rho^0 \|_p \le \| \rho^0_{N} \|_p \bigr) \le
[2(R+1)]^d N^\gamma e^{- c_R \| \rho \|_p N^{1-\gamma} },
\end{equation*}
where $L_d$ and $c_R$ are explicitly given by
\[
c_R := \frac{2\ln 2 }{\left[ 2(R+1) \right]^{\frac dp}} \quad
\mbox{and} \quad L_d := \frac{4 (4[[\sqrt{d}]]+1)^{d/p}}{c_d},
\]
with $[[\cdot]]$ denoting the integer part.
\end{proposition}

\begin{proof}
For any $X_i \in \R^{d}$ and $x \in \R^d$, we have
\[
 \rho^0_{N}(x)  =  \frac{1}{N\,c_d \,\e^d} \sum_{i=1}^{N} \mathbf{1}_{B_\e}(x - X_i) =  \frac 1{N\,c_d\,\ep^d}
 \# \{ i \text{ s.t. } |x - X_i| \leq   \e  \},
\]
where $\#$ stands for the cardinal (of a finite set). Next, we
cover $[-R,R]^d$ by $M$ disjoint cubes $C_k$ of size $\ep^d$,
centered at the points $(c_k)_{k \leq M}$. The number $M$ of
square needed depends on $N$ via $\ep$, and is bounded by
$$
M \leq \left[ \frac{2(R+1)}\e \right]^d.
$$
Assume that $x \in C_k$ for some $1\leq k \le M$, i.e., $|x - c_k|
\leq \frac{\sqrt{d}\ep}2$, then
$$
 \# \{ i \text{ s.t. } |x - X_i | \leq \e \} \leq \# \{ i \text{ s.t. } |c_k - X_i|_\infty \leq 2\sqrt{d} \e \},
$$
and for any $1<p<\infty$ we obtain
\begin{align*}
\int_{C_k} (\rho^0_{N}(x))^p \,dx &  \le
\frac{\ep^{d(1-p)}}{(N\,c_d)^{p}} \, \# \{ i \text{ s.t. } |c_k -
X_i|_\infty \leq  2\sqrt{d} \e  \}^p \\
& = \frac{\ep^{d(1-p)}}{(N\,c_d)^{p}} \, \# \{ i \text{ s.t. }
x\in C_k^d \}^p\,,
\end{align*}
where $C_k^d$ denotes the cube of center $c_k$ and size
$(4\sqrt{d} \e)^d$. Let us consider the set of cubes of the
lattice that contains $C_k^d$, i.e.,
$$
C_k^d \subset \bigcup_{j\in I_k} C_j
$$
where $I_k=\{j \mbox{ such that } C_k^d\cap C_j\neq \emptyset\}$.
It is direct to check that $\#I_k\leq M_d$ with
$M_d=(4[[\sqrt{d}]]+1)^d$. Moreover, there are only $M_d$ possible
values of $1\leq k\leq M$ such that $j\in I_k$ for a given $1\leq
j\leq M$. This yields
\begin{align}
\int_{\bbr^d} (\rho^0_{N}(x))^p dx &\leq
\frac{\ep^{d(1-p)}}{(N\,c_d)^{p}} \, \sum_{k=1}^M \sum_{j\in I_k}
\# \{ i \text{ s.t. } x\in C_j \}^p\nonumber\\ &\leq
\frac{M_d\ep^{d(1-p)}}{(N\,c_d)^{p}} \sum_{k=1}^M
 \# \{ i \text{ s.t. } X_i \in C_k \}^p \,.\label{tech}
\end{align}
Let us introduce the notation $N_k :=\# \{ i \text{ s.t. } X_i \in
C_k \}$. $N_k$ is a random variable which follows a binomial law
$B(N,s_k)$ with $s_k := \int_{C_k} \rho^0(x)\,dx $. If $L\| \rho^0
\|_p \le \| \rho^0_{N} \|_p$, then \eqref{tech} together with
H\"older's inequality imply that
$$
\sum_{k=1}^M  N_k^p \ge \frac{(c_d N)^p}{M_d} \ep^{d
\frac{p}{p^{\prime}} } \| \rho^0_{N} \|_p^p \ge N^p \tilde{L}^p
\ep^{d \frac{p}{p^{\prime}} } \| \rho^0 \|_p^p \ge N^p \tilde{L}^p
\sum_{k=1}^M s_k^p ,
$$
where $\tilde{L}:=c_d L/(M_d)^{1/p}$. But, if this happens, it
means that for at least one $k \le M$,
\begin{align*}
\begin{aligned}
N_k &\ge \left(\frac12 M^{-1} (N \tilde{L})^p \ep^{d
\frac{p}{p^{\prime}}} \| \rho^0 \|_p^p + \frac12 N^p \tilde{L}^p
s_k^p \right)^{\frac1p} \cr &\ge \frac12 M^{-\frac1p} N
\ep^{\frac{d}{p^{\prime}}} \tilde{L}\| \rho^0 \|_p + \frac12 N
\tilde{L} s_k \ge \frac {N\tilde{L}}2 \left(  \tilde{c}_R \ep^d \|
\rho^0 \|_p +  s_k \right),
\end{aligned}
\end{align*}
with $\tilde{c}_R := 1/\left[ 2(R+1) \right]^{\frac dp}$, where
the concavity of $x^{1/p}$ was used. Then, we deduce that
$$
\PP \bigl( L\| \rho^0 \|_p \le  \| \rho^0_{N} \|_p \bigr) \le
\sum_{k=1}^M \PP \Bigl( N_k \ge \frac {N\tilde{L}}2 [ \tilde{c}_R
\ep^d \| \rho^0 \|_p +  s_k ] \Bigr)\,.
$$
Since $N_k$ is a random variable which follows a binomial law
$B(N,s_k)$, then for any $\lambda$, the exponential moments of
$N_k$ are bounded by
$$
\EE(e^{\lambda N_k}) \leq \left[1+(e^\lambda -1)s_k\right]^N\leq
e^{(e^{\lambda} -1) N s_k}\,.
$$
This together with Chebyshev's inequality implies that
\begin{align*}
\PP \bigl( L\| \rho^0 \|_p \le \| \rho^0_{N} \|_p \bigr) &\le
\sum_{k=1}^M \EE(e^{\lambda N_k})  e^{- \lambda \frac
{N\tilde{L}}2 [  \tilde{c}_R \ep^d \| \rho^0 \|_p
+  s_k ]}\\
& \leq  \sum_{k=1}^M  e^{(e^{\lambda} -1) N s_k- \lambda \frac
{N\tilde{L}}2 [  \tilde{c}_R \ep^d \| \rho^0 \|_p +  s_k ] }.
\end{align*}
Taking $\lambda  = \ln L'$ with the notation
$L'=\frac{\tilde{L}}{2}$, we get
\begin{align*}
\PP \bigl( L\| \rho^0 \|_p \le \| \rho^0_{N} \|_p \bigr) &\le
\sum_{k=1}^M e^{- (L' \ln L'  + 1 -L') N s_k- L' \ln L'
\tilde{c}_R
N \ep^d \| \rho^0 \|_p  }\\
& \le \sum_{k=1}^M e^{-L' \ln L' c_R N \ep^d \| \rho^0 \|_p   } =
M e^{-L' \ln L' \tilde{c}_R N \ep^d \| \rho^0 \|_p},
\end{align*}
where we used $x\ln x -x + 1 \geq 0$, for $x > 0$. With the
scaling $\ep(N) = N^{- \frac \gamma d}$, we get
\begin{align*}
\PP \bigl( L\| \rho^0 \|_p \le \| \rho^0_{N} \|_p \bigr) & \le
[2(R+1)]^d N^\gamma e^{- \tilde{c}_R L' \ln L'  \| \rho^0 \|_p
N^{1-\gamma} }.
\end{align*}
In particular, choosing $L = L_d = 4(M_d)^{\frac{1}{p}}/c_d$ so
that $L' = 2$, we get the desired result
\[
\PP \bigl( L_d\| \rho^0 \|_p \le \| \rho^0_{N} \|_p \bigr) \le
[2(R+1)]^d N^\gamma e^{- c_R \| \rho^0 \|_p N^{1-\gamma} }\,,
\]
for $1<p<\infty$. In the case of $p = \infty$, we first notice
that as in \eqref{tech}, we deduce
\[
\|\rho^0_N\|_{\infty} \leq \frac{M_d}{Nc_d\e^d}\sup_{1\leq k \leq
M} \# \{ i \text{ s.t. } |c_k - X_i|_{\infty} \leq   \e
\}=\frac{M_d}{Nc_d\e^d}\sup_{1\leq k \leq M} N_k.
\]
Since $N_k$ follows a binomial law $B(N,s_k)$ and $s_k \leq
\|\rho^0\|_{\infty}\e^d$, above estimates allow us to conclude the
desired inequality.
\end{proof}

We are now in a position to give the proof of propagation of
chaos.

\begin{proof}[Proof of Theorem~\ref{chaos-thm}]
We introduce several sets for the random initial data:
\[
\omega_1 := \{ X^{N,0} : \eta^0_m \geq \e^{r} \},\quad \omega_2 :=
\{ X^{N,0} : L_d\|\rho^0\|_p \geq \|\rho^0_{N}\|_p \},
\]
and
\[
\omega_3 := \{ X^{N,0} : d_1(\mu^0_{N},\rho^0) \leq \e \}\,,
\]
where $r$, $\varepsilon$ and $L_d$ are given in Propositions
\ref{prop-chao}, \ref{prop:mindist}, and \ref{prop:normLp}. We
first provide the estimate of $\bbp(\omega_1^c)$. Note that since
the assumption on $\gamma$, we obtain
\[
\frac{2p-1}{\gamma(p-1)} < \frac{d}{p'(1+\alpha)}.
\]
This yields the existence of $r$ verifying
\[
1<\frac2\gamma\leq\frac{2p-1}{\gamma(p-1)} < r <
\frac{d}{p'(1+\alpha)}.
\]
This again implies the existence of $\beta > 0$ satisfying
\[
\frac{d}{\gamma}\beta + \frac{2p-1}{\gamma(p-1)} < r.
\]
From Proposition \ref{prop:mindist}, if we choose $L =
N^{-\beta},~\e = N^{-\gamma/d}$, then
\begin{align*}
\begin{aligned}
\bbp(\omega_1^c) &= \bbp\left(X^{N,0} : \eta_m^0 \leq
\e^{r}\right) = \bbp\left(X^{N,0} : \eta_m^0 \leq N^{-\frac{\gamma
r}{d}}\right)  \cr &\leq \bbp \left( X^{N,0} :\eta_m^0 \le L N^{-
\frac{2p-1}{d(p-1)}} \right)\leq 1 - e^{-2c_d^{1/p'} \|\rho^0\|_p
L^{d/p'}} \cr & \leq 2c_d^{1/p'} \|\rho^0\|_p L^{d/p'} \leq
CN^{-s},
\end{aligned}
\end{align*}
for a sufficiently large $N$ such that $N \geq (2c_d^{1/p'}
\|\rho^0\|_p)^{\frac{p'}{p' + d\beta}}$, where $s =
\frac{d\beta}{p'}$. For the estimate of $\bbp(\omega_2^c)$, we use
the result of Proposition \ref{prop:normLp} to obtain
\[
\bbp(\omega_2^c) \leq CN^{\gamma}e^{-C N^{1-\gamma}}.
\]
Finally the estimate of $\bbp(\omega_3^c)$ follows from
\cite[Proposition 1.2 of Annexe A]{Boi2} (see also \cite{Boi,BGV}) that
\[
\bbp\left( X^{N,0} : d_1(\mu^0_N,\rho^0) \ge \e \right) \leq
CN^{-s'},
\]
where $C$ and $s'$ are positive constants. We now denote $\omega
:= \omega_1 \cap \omega_2 \cap \omega_3$. Then we have
\[
\bbp(\omega^c) \leq CN^{-l},
\]
for some positive constants $C$ and $l$. If the initial data
belongs to $\omega$, then we obtain from Proposition
\ref{prop-chao} that
\[
d_1(\rho_{N}(t),\mu_N(t)) \leq d_{\infty}(\rho_{N}(t),\mu_N(t))
\leq \frac{Ce^{CT}}{N^{\gamma/d}}, \quad \mbox{for} \quad t \in
[0,T].
\]
We also notice from Theorem \ref{BLRlike-thm} that
\[
d_1\left(\rho(t), \rho_{N}(t)\right) \leq
d_1\left(\rho^0,\rho^0_{N}\right)e^{CT}\leq
(d_1\left(\rho^0,\mu_N^0\right) +
d_{\infty}\left(\mu_N^0,\rho^0_{N}\right))e^{CT}\,,
\]
for all $t \in [0,T]$. Since
$d_{\infty}\left(\mu_N^0,\rho^0_{N}\right)\leq \varepsilon$ and
the initial data belongs to $\omega$, this yields
\[
d_1\left(\rho(t), \rho_{N}(t)\right) \leq
\frac{Ce^{CT}}{N^{\gamma/d}},
\]
for all $t \in [0,T]$ since
\[
d_1\left(\rho^0,\rho^0_{N,\e}\right) \leq
d_1\left(\rho^0,\mu_N^0\right) +
d_{\infty}\left(\mu_N^0,\rho^0_{N,\e}\right) \leq
\frac{Ce^{CT}}{N^{\gamma/d}}.
\]
Hence, we have
\[
\bbp(\omega) \leq \bbp \left( \sup_{t \in [0,T]} d_1\left(\rho(t),
\rho_{N}(t)\right) \leq \frac{Ce^{CT}}{N^{\gamma/d}}\right),
\]
and it implies the desired result
\[
\bbp \left( \sup_{t \in [0,T]} d_1\left(\rho(t),
\rho_{N}(t)\right) \geq \frac{Ce^{CT}}{N^{\gamma/d}}\right) \leq
\bbp(\omega^c) \leq \frac{C}{N^{l}}\,.
\]
\end{proof}

\par

%
%
%
%
%

\end{document}